\documentclass{amsart}
\usepackage{graphicx}
 \newtheorem{thm}{Theorem}[section]
 \newtheorem{cor}[thm]{Corollary}
 \newtheorem{lem}[thm]{Lemma}
 
 \theoremstyle{definition}
 \newtheorem{defn}[thm]{Definition}
 \theoremstyle{remark}

 \numberwithin{equation}{section}

\begin{document}
\title[Toeplitz Products on $L^2_a(\mathbb{C}^n)$]{Products of Toeplitz Operators\\ on a Vector Valued Bergman Space}%

\author[Robert Kerr]{Robert Kerr}

\address{%
Department of Mathematics,\\ University of Glasgow,\\ University
Gardens,\\ Glasgow G12 8QW,\\ UK}

\email{rkerr@maths.gla.ac.uk}

\thanks{This work was completed with the support of the EPSRC. Part of this work was done while visiting the Fields Institute as part of the Thematic Program on New Trends in Harmonic Analysis}


\begin{abstract}

We give a necessary and a sufficient condition for the boundedness
of the Toeplitz product $T_FT_{G^*}$ on the vector valued Bergman
space $L_a^2(\mathbb{C}^n)$, where $F$ and $G$ are matrix symbols
with scalar valued Bergman space entries. The results generalize
those in the scalar valued Bergman space case \cite{products}. We
also characterize boundedness and invertibility of Toeplitz products
$T_FT_{G^*}$ in terms of the Berezin transform, generalizing results
found by Zheng and Stroethoff for the scalar valued Bergman space
\cite{products2}.
\end{abstract}

\maketitle
\section{Introduction}
\subsection{Notation}
 For a measurable function $f:\mathbb{D} \rightarrow
\mathbb{C}^n$ with $(\int_{\mathbb{D}}
||f(z)||^p_{\mathbb{C}^n}dA(z))^{\frac{1}{p}} < \infty$, we say that
$f \in L^p(\mathbb{D},\mathbb{C}^n)$. The vector-valued Bergman
space $L^p_a(\mathbb{D},\mathbb{C}^n)$ is the intersection of
$L^p(\mathbb{D},\mathbb{C}^n)$, with the analytic $\mathbb{C}^n$-
valued functions on $\mathbb{D}$ with the usual identification of
functions which only differ on sets of measure 0. The norm is given
by $||f||_{L^p_a(\mathbb{C}^n)} = (\int_{\mathbb{D}}
||f(z)||^p_{\mathbb{C}^n}dA(z))^{\frac{1}{p}}$, where $dA$ is
normalized Lebesgue measure on the unit disk $\mathbb{D}$. In the
case $p=2$ this space becomes a Hilbert space with the inner product
given by $\left\langle f,g \right\rangle = \int_{\mathbb{D}}
\left\langle f(z),g(z) \right\rangle_{\mathbb{C}^n}dA(z).$ $L^p$ and
$L_a^2$ are Banach spaces for $1 \leq p < \infty$. For
details see for example \cite{blasco}.\\

On the scalar valued Bergman space $L_a^2$, the Toeplitz operator
with symbol $f \in L^2$ is the densely defined operator $T_fv =
P(fv)$, where $P$ is the orthogonal projection from $L^2$ into
$L_a^2$ and $v$ is a polynomial. The Toeplitz operator is a
multiplication operator composed with an orthogonal projection. The
Bergman projection is explicitly  given by the following integral;
\[Pf(w) = \left\langle f,K_w \right\rangle = \int_{\mathbb{D}}\frac{f(w)}{(1-\overline{z}w)^2}dA(w),\]
where $K_w(z) = \frac{1}{(1-z\overline{w})^2}$ is the reproducing
kernel of the Bergman space $L_2^2(\mathbb{D})$. So using this
explicit form we can define a Toeplitz operator on a dense subset of
$L^2_a$, the polynomials, with symbol in $L^2$ rather than
$L^\infty$. We can also see that with a symbol $f \in L^2$ and $v
\in L^2_a$, $T_fv(w)$ is well defined point-wise for each $w \in
\mathbb{D}.$

In \cite{sarason} Sarason conjectured that a product of Toeplitz
operators (defined densely in an appropriate way for analytic
functions $f$ and $g$) $T_fT_{\overline{g}}$ on the Hardy Space
$H^2$ is bounded if and only if
$\widehat{|f|^2}(w)\widehat{|g|^2}(w)$ was uniformly bounded on the
disc, $\widehat{f}(w)$ being the Poisson integral of $f$. This
turned out to be false \cite{nazarov}.

Another conjecture
 by Sarason dealt with in \cite{products,park,invproducts,pottproducts,products2,products3}
and \cite{products4} was as to when the densely defined operator
$T_fT_{\overline{g}}$ is bounded on $L^2_a$
 for $f,g \in L^2_a$? The question was originaly posed by Sarason in \cite{sarason} and a conjecture in section 8 of \cite{products} more explicitly  resembles Sarason's Hardy space case conjecture.
 This time it is conjectured that the Toeplitz product $T_fT_{\overline{g}}$ is bounded for analytic $f$ and $g$ on the Bergman space $L_a^2$ if and only if
 $\widetilde{|f|^2}(w)\widetilde{|g|^2}(w)$ is uniformly bounded,
 where $\widetilde{f}$ is the Berezin transform of $f$.

 The question is investigated in various
 different cases, such as the weighted Bergman space with standard weights and the Bergman space on the unit ball and polydisk. These papers prove results that approximate to the Bergman space version of Sarason's conjecture as stated in section 8 of \cite{products}. The purpose of this paper is to investigate products of Toeplitz operators on a Bergman space of vector-valued functions. In the case of the vector valued Bergman space
 $L_a^2(\mathbb{C}^n)$,
we define the Toeplitz operator to be the densely defined
composition of multiplication with a matrix valued function and the
orthogonal projection from $L^2(\mathbb{C}^n)$ into
$L_a^2(\mathbb{C}^n)$. So in this case the symbol F will be a matrix
of $L^2$ functions and $T_Fv = P(Fv)$, where $v$ is a bounded
analytic $\mathbb{C}^n$ valued function. If
\[F = \begin{pmatrix} f_{11}  &  f_{12} & \ldots\\  f_{21} &\ddots
& \\ \vdots & \end{pmatrix}\] and $v = (v_1,v_2,\ldots,v_n)$, where
$f_{ij} \in L^2$ and $v_i \in H^\infty$, then \[T_Fv = P(Fv) =
P\left(\sum_{i=1}^n f_{1i}v_i,\sum_{i=1}^n f_{2i}v_i,
\ldots,\sum_{i=1}^n f_{ni}v_i\right) =
\begin{pmatrix}T_{f_{11}} & T_{f_{12}} & \ldots\\ T_{f_{21}} &\ddots
& \\ \vdots & \end{pmatrix}v,\] where each $T_{f_{ij}}$ is a densely
defined Toeplitz operator on the scalar Bergman space $L_a^2$. When
looking at products of these Toeplitz operators analagous to the
treatment in \cite{products} we have products of the form $T_F
T_{G^*}$, where $F$ and $G$ are square matrices of scalar valued
Bergman space $L_a^2$ functions.

\subsection{Main Theorems}
The first two main theorems follow, one giving a sufficient
condition for the Toeplitz product $T_FT_{G^*}$ to be bounded and
the other a necessary condition. Both are conditions involving the
Berezin transform;

\begin{defn} The Berezin transform of a matrix $A$ with $L^2$
entries is the matrix-valued function $B(A)$, where $B(A)(w) = \int
(A \circ \phi_w)(z)dA(z)$, $w \in \mathbb{D}$, composition here
being composition with each matrix entry. Here, $\phi_w$ is the
M\"{o}bius transform $z \mapsto \frac{w-z}{1-\overline{w}z}.$ We
should also note here that $B(A)(w) = \int
A(z)\frac{(1-|w|^2)^2}{|1-\overline{w}z|^4}dA(z)$ by a change of
variables. Defining the normalized reproducing kernel $k_w(z)$ to be
$\frac{K_w(z)}{||K_w||},$ we obtain
$|k_w(z)|^2=\frac{(1-|w|^2)^2}{|1-\overline{w}z|^4}.$
\end{defn}
Here is our first main result:

\begin{thm}\label{eps}

If for some $\epsilon > 0$ the trace of the matrix
$B((F^*F)^{\frac{2+\epsilon}{2}})(w)B((G^*G)
^{\frac{2+\epsilon}{2}})(w)$ is uniformly bounded for all $w\in
\mathbb{D}$, then the Toeplitz product $T_FT_{G^*}$ is bounded
$L_a^2(\mathbb{C}^n) \rightarrow L_a^2(\mathbb{C}^n)$.

We also have the following condition: If there exists $\epsilon > 0$
such that
\[\left(\int_\mathbb{D}\left\{\int_{\mathbb{D}}(Ctr(G(z)F(x)^*F(x)G(z)^*)
)^{\frac{2+\epsilon}{2}}|k_w(z)|^2dA(z)
\right\}|k_w(x)|^2dA(x)\right)^{\frac{1}{2+\epsilon}}\] is uniformly
bounded, then the Toeplitz product
$T_FT_{G^*}:L_a^2(\mathbb{C}^n)\rightarrow L_a^2(\mathbb{C}^n)$ is
bounded.

\end{thm}

Here is the necessary condition:
\begin{thm}\label{trace}
If the product of Toeplitz operators $T_FT_{G^*}$ is bounded, then
the trace of the matrix $B(F^*F)(w)B(G^*G)(w)$ is uniformly bounded
for  $w\in\mathbb{D}.$
\end{thm}

The next theorem is the other main result presented here, involving
a characterization of bounded and invertible Toeplitz products.

\begin{thm}\label{final}

  \item The Toeplitz product $T_FT_{G^*}$ is bounded and invertible
if and only if the trace of the matrix $B(F^*F)(w)B(G^*G)(w)$ is
uniformly bounded and there exists $\eta
> 0$ with $(FG^*GF^*)(z) >
\eta I$ for all $z \in \mathbb{D}.$ This last inequality is a matrix
inequality.

\end{thm}

\section{Bounded Toeplitz Products}
\subsection{A Sufficient Condition(Proof of Theorem \ref{eps})}

The technique in \cite{products} for showing a sufficient condition
on the boundedness of a Toeplitz product involves an inner product
formula that easily generalizes to the vector valued case. So for
$g,f \in L_a^2(\mathbb{C}^n)$

\[ \left\langle
f,g\right\rangle_{L_a^2(\mathbb{C}^n)} = \int_\mathbb{D}
\left\langle f(z),g(z) \right\rangle_{\mathbb{C}^n} dA(z) \]\[=
3\int_\mathbb{D}(1-|z|^2)^2\left\langle f(z),g(z)
\right\rangle_{\mathbb{C}^n}dA(z) +\]\[
\frac{1}{2}\int_\mathbb{D}(1-|z|^2)^2\left\langle f'(z),g'(z)\right
\rangle_{\mathbb{C}^n}dA(z) +
\frac{1}{3}\int_\mathbb{D}(1-|z|^2)^3\left\langle f'(z),g'(z)\right\rangle_{\mathbb{C}^n}dA(z).\]\\

So to estimate the norm of $T_GT_F^*$, we will look at the inner
product \\$\left\langle
T_GT_F^*u,v\right\rangle_{L^2_a(\mathbb{C}^n)}$ with $v,u \in
L^2_a(\mathbb{C}^n)$ in the form just given.

Let us start by estimating the term $\left\langle
T_{F^*}(u)(w),T_{G^*}(v)(w)\right\rangle_{\mathbb{C}^n}$.

\begin{defn}
For $f,g \in L^2(\mathbb{D})$ , define the rank 1 operator $f
\otimes g:L^2(\mathbb{D})\rightarrow L^2(\mathbb{D})$ by
\[(f \otimes g)h = \left\langle h,g\right\rangle f\] for $h \in L^2(\mathbb{D}).$

Also for $F,G \in M_{n \times n}(L^2(\mathbb{D}))$, define the
operator $F \otimes G:L^2(\mathbb{D},\mathbb{C}^n)\rightarrow
L^2(\mathbb{D},\mathbb{C}^n)$ by
\[(F \otimes G) h =
\begin{pmatrix}\sum_i f_{1i} \otimes g_{1i} & \sum_i f_{1i}\otimes
g_{2i} & \ldots& \sum_i
f_{1i} \otimes g_{ni}\ \\\sum_i f_{2i} \otimes g_{1i}&\ldots \\\vdots&\ddots  \\  &\\
\sum_i f_{ni} \otimes g_{1i} &\sum_i f_{ni} \otimes g_{2i} &\ldots
&\sum_i f_{ni} \otimes g_{ni}
\end{pmatrix}h\] for $h \in L^2(\mathbb{D},\mathbb{C}^n).$
\end{defn}

\begin{thm}\label{1st}
\[\left\langle T_{F^*}(u)(w),T_{G^*}(v)(w)\right\rangle_{\mathbb{C}^n} =
\frac{1}{(1-|w|^2)^2} \int_{\mathbb{D}}\left\langle (Gk_w \otimes
Fk_w )u(z) , v(z)\right\rangle_{\mathbb{C}^n}dA(z),\] where $k_w$ is
the normalized reproducing kernel.
\end{thm}
\begin{proof}

\[\left\langle T_{F^*}(u)(w),T_{G^*}(v)(w)\right\rangle_{\mathbb{C}^n} \]

\[=\left\langle
\int_{\mathbb{D}}F^*(z)u(z)\overline{K_w(z)}dA(z),\int_{\mathbb{D}}G^*(\zeta)v(\zeta)\overline{K_w(\zeta)}dA(\zeta)\right\rangle_{\mathbb{C}^n}
\]

\[=\int_{\mathbb{D}}\int_{\mathbb{D}}\left\langle
G(\zeta)K_w(\zeta)\left(F(z)K_w(z)\right)^*u(z),v(\zeta)\right\rangle_{\mathbb{C}^n}dA(z)dA(\zeta)
\]

\[=\frac{1}{(1-|w|^2)^2}\int_{\mathbb{D}}\int_{\mathbb{D}}\left\langle
G(\zeta)k_w(\zeta)\left(F(z)k_w(z)\right)^*u(z),v(\zeta)\right\rangle_{\mathbb{C}^n}dA(z)dA(\zeta)
\]

\[= \frac{1}{(1-|w|^2)^2}  \int_{\mathbb{D}}\left\langle (Gk_w \otimes Fk_w u)(\zeta) ,
v(\zeta)\right\rangle_{\mathbb{C}^n}dA(\zeta).\]
\end{proof}

\begin{lem}\label{traceop}
 \[||(F \otimes G) (G \otimes F)||_{op} \sim  trace\{(F \otimes G) (G \otimes F)\} = \sum_{q=1}^n
\sum_{m=1}^n\sum_{r=1}^n\sum_{l=1}^n\left\langle f_{qr},
f_{ql}\right\rangle_{L^2}\left\langle
g_{ml},g_{mr}\right\rangle_{L^2}.\]
\end{lem}

\begin{proof}

As $(F \otimes G)(G \otimes F)$ is of finite rank the trace of $(F
\otimes G) (G \otimes F)$ will be an equivalent norm. We can express
$F \otimes G$ as a matrix of operators on the scalar Bergman space
with the entries $[ \sum_{l=1}^n f_{il} \otimes g_{jl} ]_{i,j}$. We
can then express $(F \otimes G)(G \otimes F)$ in a similar manner;

\[ \left[\sum_{m=1}^n \left(\sum_{l=1}^n f_{il} \otimes g_{ml}\right)\left(\sum_{l=1}^n g_{ml} \otimes
f_{jl}\right)\right]_{i,j}\]\[=
\left[\sum_{m=1}^n\left(\sum_{r=1}^n\sum_{l=1}^n \left\langle \cdot
, f_{jl}\right\rangle_{L^2}\left\langle
g_{ml},g_{mr}\right\rangle_{L^2}f_{ir}\right)\right]_{i,j}.\] Noting
that we have as an orthonormal basis $e_{l,m} = (0, \ldots,
0,z^l\sqrt{l+1},0, \ldots )$ i.e. a vector with each coordinate $0$
apart from the mth entry which is the lth orthonormal basis element
of the scalar valued Bergman space. So the trace of the operator $(F
\otimes G)(G \otimes F)$ will be

\[ \sum_{p,q} \left\langle (F \otimes G)(G \otimes F) e_{p,q} , e_{p,q} \right\rangle \]\[= \sum_{q=1}^n \sum_{p=1}^\infty \sum_{m=1}^n\left(\sum_{r=1}^n\sum_{l=1}^n \left\langle  z^p\sqrt{1+p} ,
f_{ql}\right\rangle_{L^2}\left\langle
g_{ml},g_{mr}\right\rangle_{L^2}\int_\mathbb{D}f_{qr}(z)\overline{z^p\sqrt{1+p}}dA(z)\right).\]

We can write each $f_{ij}$ as a power series $\sum_{s=1}^\infty
a_{s,ij} z^s\sqrt{1+s}$ and thus this trace becomes

\[ \sum_{q=1}^n \sum_{p=1}^\infty \sum_{m=1}^n\left(\sum_{r=1}^n\sum_{l=1}^n \overline{a_{p,ql}}\left\langle g_{ml},g_{mr}\right\rangle_{L^2}a_{p,qr}\right)\]

and thus by Parseval's identity the expression for the trace
becomes;

\[ \sum_{q=1}^n  \sum_{m=1}^n\left(\sum_{r=1}^n\sum_{l=1}^n\left\langle f_{qr},
f_{ql}\right\rangle_{L^2}\left\langle
g_{ml},g_{mr}\right\rangle_{L^2}\right)\]
\end{proof}

\begin{thm}
$||Gk_w \otimes Fk_w||_{op} \approx (tr
(B(G^*G)(w)B(F^*F)(w)))^{\frac{1}{2}}.$

\end{thm}

\begin{proof}
Using Lemma \ref{traceop} we can see that as $||Gk_w \otimes
Fk_w||_{op}$ is equivalent to the square root of the trace of the
operator $(Gk_w \otimes Fk_w)(Fk_w \otimes Gk_w)$ and this is equal
to \[ \sum_{q=1}^n
\sum_{m=1}^n\left(\sum_{r=1}^n\sum_{l=1}^n\left\langle f_{qr},
f_{ql}|k_w|^2\right\rangle_{L^2}\left\langle
g_{ml},g_{mr}|k_w|^2\right\rangle_{L^2}\right) =
tr(B(G^*G)(w)B(F^*F)(w))\] we immediately have our result.

\end{proof}
\begin{defn}
The operator $P_0$ defined on $L_p(\mathbb{D})$ is the operator that
sends $f \in L^2$ to the function given by $(P_0f)(w) =
\int_{\mathbb{D}} \frac{f(z)}{|1-\overline{w}z|^2}dA(z).$
\end{defn}

Elements from the following two theorems are borrowed from Theorem
3.2 in \cite{products}.
\begin{lem}\label{basic}
If we have a scalar valued integrable function $h$ and a scalar
valued Bergman space function $v$ then for each $w \in \mathbb{D}$,
\[\int_{\mathbb{D}}\left|
\frac{\overline{xh(x)}|v(x)|}{(1-\overline{x}w)^3}\right|dA(x) \leq
2 \left\{\int_{\mathbb{D}}
|h(x)|^{2+\epsilon}\frac{|k_w(x)|^2}{1-|w|^2}dA(x)\right\}^{\frac{1}{2+\epsilon}}\left\{(P_0|v|^{\delta})(w)\right\}^{\frac{1}{\delta}}.\]
Here, for $\epsilon > 0$, $\delta = \frac{2+\epsilon}{1+\epsilon}$.
\end{lem}

\begin{proof}
By H\"{o}lder's inequality,
\[\int_{\mathbb{D}}\left|
\frac{\overline{xh(x)}|v(x)|}{(1-\overline{x}w)^3}\right|dA(x) \leq
\int_{\mathbb{D}}
\frac{|h(x)||1-\overline{x}w||v(x)|}{|1-\overline{x}w|^4}dA(x)\]

\[\leq\left\{\int_{\mathbb{D}}
\frac{|h(x)|^{2+\epsilon}}{|1-\overline{x}w|^4}dA(x)\right\}^{\frac{1}{2+\epsilon}}\left\{\int_{\mathbb{D}}
\frac{|1-\overline{x}w|^{\delta}|v(x)|^{\delta}}{|1-\overline{x}w|^4}dA(x)\right\}^{\frac{1}{\delta}}
 \]

\[= \left\{\int_{\mathbb{D}}
|h(x)|^{2+\epsilon}\frac{|k_w(x)|^2}{1-|w|^2}dA(x)\right\}^{\frac{1}{2+\epsilon}}
\left\{\int_{\mathbb{D}}
\frac{|1-|w|^2|^{\frac{\epsilon}{1+\epsilon}}|v(x)|^{\delta}}{|1-\overline{x}w|^2|1-\overline{x}w|^{\frac{\epsilon}{\epsilon
+ 1}}}dA(x)\right\}^{\frac{1}{\delta}},\]

and our result follows from the fact that
\[\frac{1-|w|^2}{|1-\overline{w}z|} \leq
\frac{(1-|w|)(1+|w|)}{|1-\overline{w}z|} \leq\\
\frac{(1-|w|)(1+|w|)}{|1-|w||z||} \leq
\frac{(1-|w|)(1+|w|)}{|1-|w||}\leq 1+|w|< 2.\]

\end{proof}

Let us now take a look at $\left\langle
T_{F^*}(u)'(w),T_{G^*}(v)'(w)\right\rangle_{\mathbb{C}^n}.$
\begin{thm}\label{2nd}
Let $w \in \mathbb{D}$. Then
\[
\left|\left\langle
T_{F^*}(u)'(w),T_{G^*}(v)'(w)\right\rangle_{\mathbb{C}^n}\right|\]
\begin{multline*}
 \leq
C\left(\int_\mathbb{D}\left\{\int_{\mathbb{D}}(tr(G(z)F(x)^*F(x)G(z)^*)
)^{\frac{2+\epsilon}{2}}\frac{|k_w(z)|^2}{1-|w|^2}dA(z)
\right\}\frac{|k_w(x)|^2}{1-|w|^2}dA(x)\right)^{\frac{1}{2+\epsilon}}
\\
\times\left\{(P_0||u||^{\delta}_{{\mathbb{C}}^n})(w)\right\}^{\frac{1}{\delta}}\left\{(P_0||v||_{\mathbb{C}^n}^{\delta})(w)\right\}^{\frac{1}{\delta}}
\end{multline*}
\begin{multline*}
\leq C\left(\left\{tr(B((F^*F)^{\frac{2+\epsilon}{2}})(w)B((G^*G)
^{\frac{2+\epsilon}{2}})(w))\right\}^{\frac{1}{2+\epsilon}}\right)\frac{1}{(1-|w|^2)^2}
\\
\times
\left\{(P_0||u||^{\delta}_{{\mathbb{C}}^n})(w)\right\}^{\frac{1}{\delta}}\left\{(P_0||v||_{\mathbb{C}^n}^{\delta})(w)\right\}^{\frac{1}{\delta}}
\end{multline*}

where $C$ is a constant, $\epsilon
>0$ and $\frac{1}{\delta} = 1 - \frac{1}{2+\epsilon}.$
\end{thm}

\begin{proof}
 First note that for a function $u \in
 L_a^2(\mathbb{D},\mathbb{C}^n)$

 \[ \left\langle u , K_w^{\prime}\right\rangle = u^{\prime}(w)\]

so that

\[|\left\langle T_{F^*}(u)'(w),T_{G^*}(v)'(w)\right\rangle
_{\mathbb{C}^n}| = |\int \int \left\langle
F^*(z)u(z)\overline{K_w^\prime(z)},
G^*(x)v(x)\overline{K_w^\prime(x)}\right\rangle_{\mathbb{C}^n}dA(z)dA(x)|\]

\[\leq |\int \int ||G(z)F^*(x)||_{\mathbb{C}^n op} ||u(z)|| ||v(x)||
\left|\frac{K_w(z)}{1-\overline{w}z}\right|
\left|\frac{K_w(x)}{1-\overline{w}x}\right|dA(z)dA(x)|\]

Then using Lemma \ref{basic} with $h(z) =
||G(z)F^*(x)||_{\mathbb{C}^n
op}||v(x)||\left|\frac{K_w(x)}{1-\overline{w}x}\right| $we arrive at
the following inequality,

\[|\left\langle T_{F^*}(u)'(w),T_{G^*}(v)'(w)\right\rangle
_{\mathbb{C}^n}| \]\[\leq 2
\left|\int_{\mathbb{D}}\left\{\int_{\mathbb{D}}
\left(||G(z)F^*(x)||_{\mathbb{C}^n
op}||v(x)||\left|\frac{K_w(x)}{1-\overline{w}x}\right|\right)^{2+\epsilon}
\frac{|k_w(z)|^2}{1-|w|^2}dA(z)\right\}^{\frac{1}{2+\epsilon}}\left\{(P_0||u||^{\delta})(w)\right\}^{\frac{1}{\delta}}dA(x)\right|\]
\[= \left|\int_{\mathbb{D}}\left(\left\{\int_{\mathbb{D}}
\left(||G(z)F^*(x)||_{\mathbb{C}^n op}\right)^{2+\epsilon}
\frac{|k_w(z)|^2}{1-|w|^2}dA(z)\right\}^{\frac{1}{2+\epsilon}}||v(x)||\left|\frac{K_w(x)}{1-\overline{w}x}\right|\right)dA(x)\left\{(P_0||u||^{\delta})(w)\right\}^{\frac{1}{\delta}}\right|\]
where $\epsilon > 0$ and $\frac{1}{\delta} = 1 -
\frac{1}{2+\epsilon}.$

Again using Lemma \ref{basic} but this time with $h(x) =
\left\{\int_{\mathbb{D}} \left(||G(z)F^*(x)||_{\mathbb{C}^n
op}\right)^{2+\epsilon}
\frac{|k_w(z)|^2}{1-|w|^2}dA(z)\right\}^{\frac{1}{2+\epsilon}}$

to see that

\[|\left\langle T_{F^*}(u)'(w),T_{G^*}(v)'(w)\right\rangle
_{\mathbb{C}^n}| \]\[\leq
4\left|\left\{\int_{\mathbb{D}}\left\{\int_{\mathbb{D}}
\left(||G(z)F^*(x)||_{\mathbb{C}^n op}\right)^{2+\epsilon}
\frac{|k_w(z)|^2}{1-|w|^2}dA(z)\right\}\frac{|k_w(x)|^2}{1-|w|^2}dA(x)\right\}^{\frac{1}{2+\epsilon}}\left\{(P_0||v||^{\delta})(w)\right\}^{\frac{1}{\delta}}\left\{(P_0||u||^{\delta})(w)\right\}^{\frac{1}{\delta}}\right|\]

\[\leq 4
\left(\int_\mathbb{D}\left\{\int_{\mathbb{D}}(Ctr(G(x)F(z)^*F(z)G(x)^*)
)^{\frac{2+\epsilon}{2}}\frac{|k_w(z)|^2}{1-|w|^2}dA(z)
\right\}\frac{|k_w(x)|^2}{1-|w|^2}dA(x)^{\frac{1}{2+\epsilon}}\right)\]

\[\times\left\{(P_0||u||^{\delta}_{{\mathbb{C}}^n})(w)\right\}^{\frac{1}{\delta}}\left\{(P_0||v||_{\mathbb{C}^n}^{\delta})(w)\right\}^{\frac{1}{\delta}}\]

(This is what we want but we can go a step further and get something
that looks even more similar to the analogous result in the scalar
case.) Letting the $4$ be absorbed into the constant $C$ and using
the inequality on matrix norms from \cite{Bhatia}, theorem IX.2.10
on page 258, we see that

\[\leq
\int_\mathbb{D}\left\{\int_{\mathbb{D}}(C||(G(x)F(z)^*F(z)G(x)^*)||_{\mathbb{C}^nop}
)^{\frac{2+\epsilon}{2}}\frac{|k_w(z)|^2}{1-|w|^2}dA(z)
\right\}\]\[\times\frac{|k_w(x)|^2}{1-|w|^2}dA(x)^{\frac{1}{2+\epsilon}}\left\{(P_0||u||^{\delta}_{{\mathbb{C}}^n})(w)\right\}^{\frac{1}{\delta}}\left\{(P_0||v||_{\mathbb{C}^n}^{\delta})(w)\right\}^{\frac{1}{\delta}}\]
\[
\leq
\int_\mathbb{D}\left\{\int_{\mathbb{D}}(C||((G^*G(x))^{\frac{1}{2}}F^*F(z)(G^*G(x))^{\frac{1}{2}})||_{\mathbb{C}^nop}
)^{\frac{2+\epsilon}{2}}\frac{|k_w(z)|^2}{1-|w|^2}dA(z)
\right\}\]\[\times\frac{|k_w(x)|^2}{1-|w|^2}dA(x)^{\frac{1}{2+\epsilon}}\left\{(P_0||u||^{\delta}_{{\mathbb{C}}^n})(w)\right\}^{\frac{1}{\delta}}\left\{(P_0||v||_{\mathbb{C}^n}^{\delta})(w)\right\}^{\frac{1}{\delta}}
\]

\[\leq
\int_\mathbb{D}\left\{\int_{\mathbb{D}}C||(G^*G(x))^{\frac{2+\epsilon}{4}}(F^*F(z))^{\frac{2+\epsilon}{2}}(G^*G(x))^{\frac{2+\epsilon}{4}}||_{\mathbb{C}^nop}
\frac{|k_w(z)|^2}{1-|w|^2}dA(z)\right\}\]\[\times\frac{|k_w(x)|^2}{1-|w|^2}dA(x)^{\frac{1}{2+\epsilon}}\left\{(P_0||u||^{\delta}_{{\mathbb{C}}^n})(w)\right\}^{\frac{1}{\delta}}\left\{(P_0||v||_{\mathbb{C}^n}^{\delta})(w)\right\}^{\frac{1}{\delta}}\]

\[\leq
\left(Ctr\left(B\left((F^*F)^{\frac{2+\epsilon}{2}}\right)(w)B\left((G^*G)
^{\frac{2+\epsilon}{2}}\right)(w)\right)^{\frac{1}{2+\epsilon}}\right)\frac{1}{(1-|w|^2)^2}
\left\{(P_0||u||^{\delta}_{{\mathbb{C}}^n})(w)\right\}^{\frac{1}{\delta}}\]\[\times\left\{(P_0||v||_{\mathbb{C}^n}^{\delta})(w)\right\}^{\frac{1}{\delta}}\]

where $B$ is the Berezin transform and $C$ is a constant that is
possibly different from line to line.
\end{proof}

Now let us use the estimates from theorems \ref{1st} and \ref{2nd}
in the inner product formula. Taking our inner product formula
\[ \left\langle T_{F^*}(u),T_{G^*}(v)\right\rangle _{L_a^2(\mathbb{C}^n)} =
\int_\mathbb{D} \left\langle T_{F^*}(u),T_{G^*}(v)\right\rangle
_{\mathbb{C}^n} dA(z) \]\[= 3\int_\mathbb{D}(1-|z|^2)^2\left\langle
T_{F^*}(u),T_{G^*}(v)\right\rangle _{\mathbb{C}^n}dA(z) +
\frac{1}{2}\int_\mathbb{D}(1-|z|^2)^2\left\langle
T_{F^*}'(u),T_{G^*}'(v)\right\rangle _{\mathbb{C}^n}dA(z) +\]\[
\frac{1}{3}\int_\mathbb{D}(1-|z|^2)^3\left\langle
T_{F^*}'(u),T_{G^*}'(v)\right\rangle _{\mathbb{C}^n}dA(z),\]

let's take the term
$\frac{1}{2}\int_\mathbb{D}(1-|z|^2)^2\left\langle
T_{F^*}'(u),T_{G^*}'(v)\right\rangle _{\mathbb{C}^n}dA(z)$ and
estimate its modulus;
\[|\frac{1}{2}\int_\mathbb{D}(1-|z|^2)^2\left\langle T_{F^*}'(u),T_{G^*}'(v)\right\rangle _{\mathbb{C}^n}dA(z)|
\]

\[\leq\frac{1}{2}\int_\mathbb{D}|\left(Ctr\left(B\left((F^*F)^{\frac{2+\epsilon}{2}}\right)(w)B\left((G^*G)
^{\frac{2+\epsilon}{2}}\right)(w)\right)\right)^{\frac{1}{2+\epsilon}})\frac{(1-|w|^2)^2}{(1-|w|^2)^2}\times\]\[
\left\{P_0||u||^{\delta}_{{\mathbb{C}}^n}(w)\right\}^{\frac{1}{\delta}}\left\{P_0||v||_{\mathbb{C}^n}^{\delta}(w)\right\}^{\frac{1}{\delta}}dA(w)|\]

\[\leq \frac{1}{2}\sup_{w\in\mathbb{D}}\left(Ctr\left(B\left((F^*F)^{\frac{2+\epsilon}{2}}\right)(w)B\left((G^*G)
^{\frac{2+\epsilon}{2}}\right)(w)\right)\right)^{\frac{1}{2+\epsilon}})\]\[\times
\int_{\mathbb{D}}\left\{P_0||u||^{\delta}_{{\mathbb{C}}^n}(w)\right\}^{\frac{1}{\delta}}\left\{P_0||v||_{\mathbb{C}^n}^{\delta}(w)\right\}^{\frac{1}{\delta}}dA(w)\]

By Cauchy-Schwarz, this expression will be less than or equal to

\[ \frac{1}{2}\sup_{w\in\mathbb{D}}\left(Ctr\left(B\left((F^*F)^{\frac{2+\epsilon}{2}}\right)(w)B\left((G^*G)
^{\frac{2+\epsilon}{2}}\right)(w)\right)\right)^{\frac{1}{2+\epsilon}}
\]\[\times\left\{\int_{\mathbb{D}}\left\{P_0||u||^{\delta}_{{\mathbb{C}}^n}(w)\right\}^{\frac{2}{\delta}}dA(w)\right\}^{\frac{1}{2}}\left\{\int_\mathbb{D}\left\{P_0||v||_{\mathbb{C}^n}^{\delta}(w)\right\}^{\frac{2}{\delta}}dA(w)\right\}^{\frac{1}{2}}\]
now as the operator $P_0$ is $L^p$ bounded for $p > 1$,
\cite{bergman1}, this expression will be less than or equal to

\[\frac{1}{2}\sup_{w\in\mathbb{D}}\left(Ctr\left(B\left((F^*F)^{\frac{2+\epsilon}{2}}\right)(w)B\left((G^*G)
^{\frac{2+\epsilon}{2}}\right)(w)\right)\right)^{\frac{1}{2+\epsilon}})
||u||_{L^2(\mathbb{C}^n)}||v||_{L^2(\mathbb{C}^n)}.\] Estimating the
term \[ \frac{1}{3}\int_\mathbb{D}(1-|z|^2)^3\left\langle
T_{F^*}'(u),T_{G^*}'(v)\right\rangle _{\mathbb{C}^n}dA(z)\] from the
inner product formula is similar.

Finally let us estimate $3\int_\mathbb{D}(1-|z|^2)^2\left\langle
T_{F^*}(u),T_{G^*}(v)\right\rangle _{\mathbb{C}^n}dA(z),$ We can see
from \ref{1st} that
\[|\int_\mathbb{D}(1-|z|^2)^2\left\langle T_{F^*}(u),T_{G^*}(v)\right\rangle _{\mathbb{C}^n}dA(z)|
 \]\[=|\int_\mathbb{D}\frac{(1-|w|^2)^2}{(1-|w|^2)^2}
\int_{\mathbb{D}}\left\langle Gk_w \otimes Fk_w u , v\right\rangle
_{\mathbb{C}^n}dA(z)dA(w)| \]

\[\leq|\int_\mathbb{D}
||Gk_w \otimes Fk_w||_{op}dA(w) ||u||_{L_a^2{\mathbb{C}^n}}
||v||_{L_a^2{\mathbb{C}^n}}|\]
\[\leq\sup_{w\in\mathbb{D}} (tr B(G^*G)(w)B(F^*F)(w))^{\frac{1}{2}} ||u||_{L_a^2{\mathbb{C}^n}}
||v||_{L_a^2{\mathbb{C}^n}}.\]

Now we just use H\"{o}lder's inequality to get an expression similar
to the one in the previous estimate.
\[tr (B(G^*G)(w)B(F^*F)(w))  \]\[=tr (\int_\mathbb{D} G(x)^*G(x)
|k_w(x)|^2dA(x)\int_\mathbb{D} F(z)^*F(z) |k_w(z)|^2dA(z))\]

\[=\int_\mathbb{D} \int_\mathbb{D} tr (G(x)^*G(x)
|k_w(x)|^2F(z)^*F(z) |k_w(z)|^2)dA(x)dA(z)  \]

\[=\int_\mathbb{D} \int_\mathbb{D} tr \left\{G(x)^*G(x)
F(z)^*F(z) \right\}|k_w(z)|^2|k_w(x)|^2dA(x)dA(z)\]

\[\leq\left\{\int_\mathbb{D} \int_\mathbb{D} (tr \left\{G(x)^*G(x)
F(z)^*F(z)
\right\})^{\frac{2+\epsilon}{2}}|k_w(z)|^2|k_w(x)|^2dA(x)dA(z)\right\}^{\frac{2}{2+\epsilon}}\]
by H\"{o}lder.

This is then less than or equal to

\[\left\{\int_\mathbb{D} \int_\mathbb{D} (Ctr \left\{(G(x)^*G(x))^{\frac{2+\epsilon}{2}}
(F(z)^*F(z))^{\frac{2+\epsilon}{2}}
\right\})|k_w(z)|^2|k_w(x)|^2dA(x)dA(z)\right\}^{\frac{2}{2+\epsilon}}\]
by Theorem IX.2.10 on page 258 of \cite{Bhatia} and similar steps as
before. This final expression is then equal to
\[\left\{tr(\int_\mathbb{D} \int_\mathbb{D}
((G(x)^*G(x))^{\frac{2+\epsilon}{2}}
(F(z)^*F(z))^{\frac{2+\epsilon}{2}}
)|k_w(z)|^2|k_w(x)|^2dA(x)dA(z))\right\}^{\frac{2}{2+\epsilon}}\]

\[=\left\{tr(B\left\{(G(x)^*G(x))^{\frac{2+\epsilon}{2}}\right\}
(w)B\left\{(F(z)^*F(z))^{\frac{2+\epsilon}{2}}\right\}
(w))\right\}^{\frac{2}{2+\epsilon}}.\]

Note that here we can use the same reasoning to see that if
\[(\int_\mathbb{D}\left\{\int_{\mathbb{D}}(Ctr(G(z)F(x)^*F(x)G(z)^*)
)^{\frac{2+\epsilon}{2}}|k_w(z)|^2dA(z)
\right\}|k_w(x)|^2dA(x)^{\frac{1}{2+\epsilon}})\] is uniformly
bounded for some $\epsilon > 0$ then our Toeplitz product
$T_FT_{G^*}$ will be bounded. This condition is seemingly stronger
and less aesthetic than the other one but it will be used later on
when dealing with Toeplitz products that are also invertible. Note
that these last inequalities show that the sufficient condition is
stronger than the necessary condition.

\subsection{A Necessary Condition}

\begin{proof}[Proof of Theorem \ref{trace}.]In \cite{park} (see also
\cite{products} for a different approach) Park shows that for
functions $f$ and $g$ in the scalar Bergman space $L_a^2$ the
operator $f \otimes g$ defined by $f \otimes g h = \left\langle
h,g\right\rangle f$ with $h \in L_a^2$ is equal to the following,
\[T_fT_{\overline{g}} - 2T_zT_fT_{\overline{g}}T_{\overline{z}} +
T_z^2T_fT_{\overline{g}}T_{\overline{z}}^2.\] Using this result in
the vector valued case, we can see that \[F \otimes G=
\begin{pmatrix}\sum f_{1i} \otimes g_{1i} & \sum f_{1i}\otimes
g_{2i} & \ldots\\ \sum f_{2i}\otimes g_{1i} &\ddots & \\ \vdots &
\end{pmatrix}= T_FT_{G^*} - 2T_zT_FT_{G^*}T_{\overline{z}} +
T_z^2T_FT_{G^*}T_{\overline{z}}^2.\]

Let us estimate the norm of the operator $(F\circ\phi_w) \otimes
(G\circ\phi_w)$, where $F\circ\phi_w$ is the matrix-valued function
\[
\begin{pmatrix}f_{11}\circ\phi_w & f_{12}\circ\phi_w &\ldots \\ f_{21}\circ\phi_w & \ddots\\
\vdots\end{pmatrix}.\]
  Noting that the operator
$(F\circ\phi_w) \otimes (G\circ\phi_w$) is of finite rank we can as
an equivalent norm take the square root of the trace of the operator
$(F\circ\phi_w \otimes G\circ\phi_w)(G\circ\phi_w \otimes
F\circ\phi_w)$ by Lemma \ref{traceop}.

Also by Lemma \ref{traceop} we can see that this will be equal to
\[ \sum_{q=1}^n  \sum_{m=1}^n(\sum_{r=1}^n\sum_{l=1}^n\left\langle f_{qr}\circ\phi_w,
f_{ql}\circ\phi_w\right\rangle _{L^2}\left\langle
g_{ml}\circ\phi_w,g_{mr}\circ\phi_w\right\rangle _{L^2})\]

\[= \sum_{q=1}^n  \sum_{m=1}^n(\sum_{r=1}^n\sum_{l=1}^n B(f_{qr}
\overline{f_{ql}})(w)B(g_{ml}\overline{g_{mr}})(w)),\]

which is equal to the trace of the matrix $B(F^*F)(w)B(G^*G)(w).$

Let $U_w$ be the unitary operator on our vector valued $L^2$ space
given by $U_w f = (f \circ \phi_w)k_w.$ It is well known that $T_{F
\circ \phi_w}U_w = U_wT_F.$ So $T_{F \circ \phi_w} = U_wT_FU_w^*,$
and thus
\begin{align*}
&\left\{tr\left(B(F^*F)(w)B(G^*G)(w)\right)\right\}^{\frac{1}{2}} =
C||F\circ\phi_w \otimes G\circ\phi_w||_{op}
\\&
= ||T_{F\circ \phi_w}T_{G^*\circ \phi_w} - 2T_zT_{F\circ
\phi_w}T_{G^*\circ \phi_w}T_{\overline{z}} + T_z^2T_{F\circ
\phi_w}T_{G^*\circ \phi_w}T_{\overline{z}}^2||_{op}
\\
&=||(U_wT_FU_w^*U_wT_{G^*}U_w^* -
2T_zU_wT_FU_w^*U_wT_{G^*}U_w^*T_{\overline{z}} +
T_z^2U_wT_FU_w^*U_wT_{G^*}U_w^*T_{\overline{z}}^2)||_{op}
\\
&=||(U_wT_FT_{G^*}U_w^* - 2T_zU_wT_FT_{G^*}U_w^*T_{\overline{z}} +
T_z^2U_wT_FT_{G^*}U_w^*T_{\overline{z}}^2)||_{op}
\\&=||(U_wT_FT_{G^*}U_w^* -
2U_wT_{\phi_w}T_FT_{G^*}T_{\overline{\phi}_w}U_w^* +
U_wT_{\phi_w}^2T_FT_{G^*}T_{\overline{\phi}_w}^2U_w^*)||_{op}.
\end{align*}
We can now use the triangle inequality on the operator\\
$U_w(T_FT_{G^*} - 2T_{\phi_w}T_FT_{G^*}T_{\overline{\phi}_w} +
T_{\phi_w}^2T_FT_{G^*}T_{\overline{\phi}_w}^2)U_w^*$ as in
\cite{products} to get our result, using that $||T_{\phi_w}|| \leq
1$.
\end{proof}

In the following we will be working with square matrices $F$ and $G$
with entries from the scalar valued Bergman space
$L_a^2(\mathbb{D})$. Where it is not explicitly stated otherwise,
this will be the case. When we refer to a matrix being less than
another matrix, $F < G$, we mean in the sense of L\"{o}wner partial
ordering of matrices. See \cite{Bhatia}, \cite{matrix1} and
\cite{matrix2} for details on this.

\section{Bounded and Invertible Toeplitz Products}
\subsection{A reverse H\"{o}lder inequality} We will now develop
some of the theory needed to show a reverse H\"{o}lder inequality
used to characterize the matrices of analytic functions, $F$ and
$G$, such that the Toeplitz product $T_FT_{G^*}$ is bounded and
invertible on the vector valued Bergman space. Compare this next
lemma with Lemma 4.6 in \cite{products2}.
\begin{lem}{\label{berezin}}
If $F(z)$ is invertible for all $z \in \mathbb{D},$ then
\[ (F^*(w)F(w))
\leq  B(F^*F)(w),\]
\[ (F^*(z)F(z))
\leq \eta_s B(F^*F)(w)\] and
\[ (F^{-1}(w)F^{*-1}(w))
\leq  B(F^{-1}F^{*-1})(w)\]
\[ (F^{-1}(z)F^{-1*}(z))
\leq \eta_s B(F^{-1}F^{*-1})(w),\]

 when $z \in D(w,s)$ the pseudohyperbolic disk with radius $0 < s < 1$
and centre $w$ and $\eta_s$ is a constant dependent only on $s$.

\end{lem}
\begin{proof}
Let $\textbf{e}$ be an arbitrary vector. Then for $F \in
L^2_a(\mathbb{C}^n)$,

\[\left\langle F(u)\textbf{e},F(u)\textbf{e}\right\rangle  = \left\langle \int F(z)\overline{K_u(z)}dA(z) \textbf{e},\int
F(z)\overline{K_u(z)}dA(z) \textbf{e}\right\rangle  \]

\[=||\int F(z)\overline{K_u(z)}dA(z)\textbf{e}||^2_{\mathbb{C}^n}
\]
\[\leq\int \left\langle F\textbf{e},F\textbf{e}\right\rangle  dA(z)||K_u||^2_{L^2} = \left\langle \int F^*(z)F(z)
dA(z) \textbf{e}, \textbf{e}\right\rangle ||K_u||^2_{L^2}.\]

 So if $u \in D(0,s)$ \[F^*(u)F(u) \leq
\int F^*(z)F(z)dA(z)||K_u||^2_2 \leq \int
F^*(x)F(x)dA(z)\frac{1}{(1-s^2)^2} .\] If $z \in D(w,s)$ then $z =
\phi_w(u)$ for some $u \in D(0,s)$ thus \[F^*(z)F(z) =
F^*(\phi_w(u))F(\phi_w(u)) \leq \int
F^*(\phi_w(x))F(\phi_w(x))dA(z)\frac{1}{(1-s^2)^2} \]\[=
B(F^*F)(w)\frac{1}{(1-s^2)^2}.\]

Now let us show that $F^{-1}(w)F^{*-1}(w) \leq B(F^{-1}F^{*-1})(w).$

\[\left\langle F^{*-1}(w)\textbf{e},F^{*-1}(w)\textbf{e}\right\rangle  =
\left\langle
F^{*-1}(\phi_w(0))\textbf{e},F^{*-1}(\phi_w(0))\textbf{e}\right\rangle
\]

\[=\left\langle \int F^{-1}(\phi_w(z))dA(z)\int
F^{*-1}(\phi_w(z))dA(z)\textbf{e},\textbf{e}\right\rangle \] and we
arrive at the conclusion that $F^{-1}(w)F^{*-1}(w) \leq
B(F^{-1}F^{*-1})(w)$ in a similar manner to before.

So for $z \in D(w,s)$ we know that $F^{-1}(w)F^{*-1}(w) \leq
B(F^{-1}F^{*-1})(w)$ and $F^*(z)F(z) \leq
B(F^*F)(w)\frac{1}{(1-s^2)^2}.$ The other inequalities follow from
applying the same procedure to $F^{-1}F^{*-1}$ instead of $F^*F.$

\end{proof}

\begin{lem}
If there exists $\eta$ such that $F(z)G(z)^*G(z)F(z)^* > \eta I$ for
all $z \in \mathbb{D}$ and $tr(B(G^*G)(w)B(F^*F)(w))$ is uniformly
bounded on $\mathbb{D}$, then
\[||B(F^{-1}(F^*)^{-1})(w))^{\frac{1}{2}}B(F^*F)(w)^{\frac{1}{2}}||\] is uniformly bounded on $\mathbb{D}$.
\end{lem}

\begin{proof}

Let us suppose that $F(w)G(w)^*G(w)F(w)^* > \eta I$ for all $w \in
\mathbb{D}$. Then $B(G^*G)(w) \geq G(w)^*G(w) \geq \eta
(F(w)^*F(w))^{-1}$. The key inequality here is \\$G(w)^*G(w) \geq
\eta (F(w)^*F(w))^{-1}$, as this implies that \\$B(G^*G)(w) \geq
\eta B((F^*F)^{-1})(w)$ and so

\begin{multline*} (B(F^*F)(w))^{\frac{1}{2}}B(G^*G)(w)(B(F^*F)(w))^{\frac{1}{2}}
\\
\geq\eta(B(F^*F)(w))^{\frac{1}{2}}B((F^*F)^{-1}(w))(B(F^*F)(w))^{\frac{1}{2}}.
\end{multline*}

  Thus as
$||(B(F^*F)(w))^{\frac{1}{2}}B(G^*G)(w)(B(F^*F)(w))^{\frac{1}{2}}||
< M$ for all $w,$

\begin{multline*} tr(B(F^{-1}(F^*)^{-1})(w)B(F^*F)(w)) \\ \leq\frac{1}{\eta}
C||(B(F^*F)(w))^{\frac{1}{2}}B(G^*G)(w)(B(F^*F)(w))^{\frac{1}{2}}||
< \eta C M
\end{multline*}

and so
\begin{multline*}
||B(F^{-1}(F^*)^{-1})(w))^{\frac{1}{2}}B(F^*F)(w)^{\frac{1}{2}}||^{2}
\\=||(B(F^{-1}(F^*)^{-1})(w))^{\frac{1}{2}}B(F^*F)(w)(B(F^{-1}(F^*)^{-1})(w))^{\frac{1}{2}}||
\\ \leq Ctr(B(F^{-1}(F^*)^{-1})(w)B(F^*F)(w)) \\ \leq \frac{1}{\eta}
C||(B(F^*F)(w))^{\frac{1}{2}}B(G^*G)(w)(B(F^*F)(w))^{\frac{1}{2}}||
\end{multline*}

where $C$ and $\eta$ are constants independent of $w$.

\end{proof}

\begin{defn}
A dyadic rectangle $Q_{j,k,l}$ is a subset of the unit disk of the
form
\[\left\{z = re^{i\theta} : (k-1)2^{-j} \leq r \leq k2^{-j},
(l-1)2^{1-j}\pi \leq \theta \leq l2^{1-j}\pi\right\},\] where
$j,k,l$ are non negative integers and $k,l \leq 2^j.$
\end{defn}

\begin{figure}[htp]
  \includegraphics[width=130pt,totalheight=0.5\textheight,viewport=50 260 350 860,]{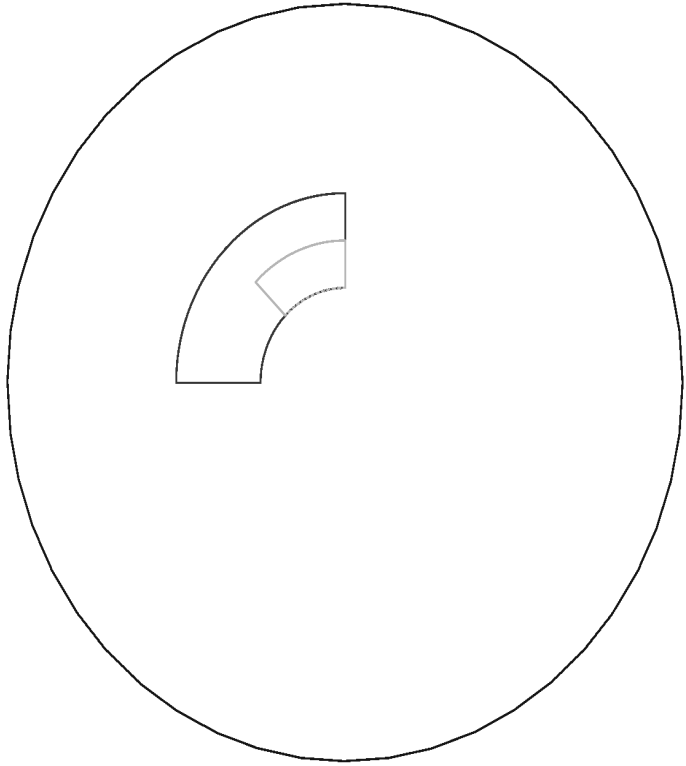}\\
  \caption{Two nested dyadic rectangles in the unit disk.}\label{major}
\end{figure}

\begin{lem}\label{contain}
There exists $0 < r < 1$ such that for all dyadic rectangles $Q$
with positive distance to the boundary  $Q \subset D(z_Q,r)$. Here,
$D$ is the pseudohyperbolic disk and $z_Q$ is the centre of the
dyadic rectangle $Q$.
\end{lem}
\begin{proof}
This is just Proposition 4.7 in \cite{products2}.
\end{proof}

Compare this next lemma with Lemma 4.8 in \cite{products2}.
\begin{lem}\label{ap}
If

\[\sup_{w \in \mathbb{D}}||B(F^{-1}(F^*)^{-1})(w))^{\frac{1}{2}}B(F^*F)(w)^{\frac{1}{2}}||
< \infty,\]

 then

 \[\sup_{Q:dyadic}||
\left\{\frac{1}{|Q|}\int_Q (F^*F) dA(z)\right\}^{\frac{1}{2}}
\left\{\frac{1}{|Q|}\int_Q (F^{-1}F^{*-1})
dA(z)\right\}^{\frac{1}{2}}|| < \infty.\]
\end{lem}

\begin{proof}
If the dyadic rectangle $Q$ is the whole disk, then as
$\int_{\mathbb{D}} F^*F dA(z) = B(F^*F)(0)$ and $\int_{\mathbb{D}}
F^{-1}F^{*-1} dA(z) = B(F^{-1}F^{*-1})(0),$ we see that

\begin{multline*} ||\left(\int_{\mathbb{D}} F^*F
dA(z)\right)^{\frac{1}{2}}\left(\int_{\mathbb{D}} F^{-1}F^{*-1}
dA(z)\right)^{\frac{1}{2}}|| \\=
||B(F^{-1}(F^*)^{-1})(0)^{\frac{1}{2}}B(F^*F)(0)^{\frac{1}{2}}||
\end{multline*}

Now let us suppose that our dyadic rectangle $Q$ has a positive
distance from the boundary. By Lemma \ref{contain} our rectangle $Q$
will be strictly contained in a pseudohyperbolic disk $D(z_Q,R)$,
$z_Q$ being the centre of our dyadic rectangle and $R$ being the
same for each dyadic rectangle. Thus by Lemma \ref{berezin}
\[ (F^{-1}(z)F^{-1*}(z))
\leq \eta B(F^{-1}F^{*-1})(z_Q)\] and \[ (F^*(z)F(z)) \leq \eta
B(F^*F)(z_Q)\] for all $z$ in our pseudohyperbolic disk $D(z_Q,R)$.
Here the constant $\eta$ will only be dependent on $R$ which is the
same for all of these dyadic rectangles.

Thus using the fact that if $A,B$ and $C$ are positive matrices such
that $A \leq B$ then $C^{\frac{1}{2}}AC^{\frac{1}{2}} <
C^{\frac{1}{2}}BC^{\frac{1}{2}}$ and
$tr(C^{\frac{1}{2}}AC^{\frac{1}{2}}) <
tr(C^{\frac{1}{2}}BC^{\frac{1}{2}}),$ we can deduce the following
series of inequalities from our hypothesis: \[||
\left\{\frac{1}{|Q|}\int_Q (F^*F) dA(z)\right\}^{\frac{1}{2}}
\left\{\frac{1}{|Q|}\int_Q (F^{-1}F^{*-1})
dA(z)\right\}^{\frac{1}{2}}||^2\]
\begin{align*}
&= ||\left\{\frac{1}{|Q|}\int_Q (F^*F) dA(z)\right\}^{\frac{1}{2}}
\left\{\frac{1}{|Q|}\int_Q (F^{-1}F^{*-1})
dA(z)\right\}\left\{\frac{1}{|Q|}\int_Q (F^*F)
dA(z)\right\}^{\frac{1}{2}}||
\\
&\leq C tr\left(\left\{\frac{1}{|Q|}\int_Q (F^*F)
dA(z)\right\}^{\frac{1}{2}} \left\{\frac{1}{|Q|}\int_Q (F^*F)^{-1}
dA(z)\right\}\left\{\frac{1}{|Q|}\int_Q (F^*F)
dA(z)\right\}^{\frac{1}{2}}\right) \\
&\leq C tr\left(\left\{\frac{1}{|Q|}\int_Q (F^*F)
dA(z)\right\}^{\frac{1}{2}} \left\{\eta B(F^{-1}F^{*-1})(z_Q)
\right\}\left\{\frac{1}{|Q|}\int_Q (F^*F)
dA(z)\right\}^{\frac{1}{2}}\right)\\
&=C tr\left(\left\{\eta (B(F^{-1}F^{*-1})(z_Q))
\right\}^{\frac{1}{2}}\left\{\frac{1}{|Q|}\int_Q (F^*F)
dA(z)\right\} \left\{\eta (B(F^{-1}F^{*-1})(z_Q))
\right\}^{\frac{1}{2}}\right) \\
&\leq C tr\left(\left\{\eta (B(F^{-1}F^{*-1})(z_Q))
\right\}^{\frac{1}{2}}\left\{\eta B(F^*F)(z_Q)\right\} \left\{ (\eta
B(F^{-1}F^{*-1})(z_Q)) \right\}^{\frac{1}{2}}\right)\\
&\leq C ||\left(\left\{\eta (B(F^{-1}F^{*-1})(z_Q))
\right\}^{\frac{1}{2}}\left\{\eta B(F^*F)(z_Q)\right\} \left\{
\eta(B(F^{-1}F^{*-1})(z_Q)) \right\}^{\frac{1}{2}}\right)||\\
&\leq C^2 \eta^4 ||\left\{ (B(F^{-1}F^{*-1})(z_Q))
\right\}^{\frac{1}{2}}\left\{B(F^*F)(z_Q)\right\}^{\frac{1}{2}}||^{2}
< M. \end{align*}
 Note that $C$ is a constant that possibly changes
from line to line and is dependent on the dimension of
$\mathbb{C}^n$ only. $M$ will be dependent only on the uniform bound
of $B(F^{-1}F^{*-1})(w) B(F^*F)(w)$, the dimension we are working in
and the constant $R$ which is the same for each dyadic rectangle not
touching the boundary.
\\
What happens when we have a dyadic rectangle that touches the
boundary but is not the whole disk? We can see that the centre of
the rectangle $z_Q$ is at a distance of at least $1/2$ from the
centre, i.e. $|z_Q| \geq 1/2$. Then
\[B(F^*F)(z_Q) = \int_{\mathbb{D}}F^*(z)F(z)|k_{z_Q}|^2(z)dA(z)
\]\[\geq \int_Q F^*(z)F(z)|k_{z_Q}(z)|^2dA(z) \geq
\frac{c}{(1-|z_Q|)^2}\int_Q F^*(z)F(z)dA(z)\] by Lemma 4.5 in
\cite{products2}.

We can also see in this case that $|Q| = 8|z_Q|(1-|z_Q|)^2$ and so

\[B(F^*F)(z_Q) \geq \frac{4c}{|Q|}\int_Q F^*(z)F(z)dA(z).\]

We can do the same for $F^{-1}F^{*-1}$ to get that

\[B(F^{-1}F^{*-1})(z_Q) \geq \frac{4c}{|Q|}\int_Q F^{-1}(z)F^{*-1}(z)dA(z).\]

We can then combine these and take the trace to see that

\[tr\left( \left\{\frac{1}{|Q|}\int_Q F^*F(z)dA(z)\right\}^{\frac{1}{2}} \frac{1}{|Q|}\int_Q (F^*F)^{-1}(z)dA(z)\left\{\frac{1}{|Q|}\int_Q
F^*F(z)dA(z)\right\}^{\frac{1}{2}}\right) \]

\[\leq 4c tr\left(\left\{B(F^{-1}F^{*-1})(z_Q)\right\}^{\frac{1}{2}}\frac{1}{|Q|}\int_Q
F^*(z)F(z)dA(z)\left\{B(F^{-1}F^{*-1})(z_Q)\right\}^{\frac{1}{2}}\right)
\]

\[\leq 16c^2
tr\left(\left\{B(F^{-1}F^{*-1})(z_Q)\right\}^{\frac{1}{2}}B(F^*F)(z_Q)\left\{B(F^{-1}F^{*-1})(z_Q)\right\}^{\frac{1}{2}}\right)\]

\[\leq C
||\left\{B(F^{-1}F^{*-1})(z_Q)\right\}^{\frac{1}{2}}\left\{B(F^*F)(z_Q)\right\}^{\frac{1}{2}})||^{\frac{1}{2}}
< M^{\prime},\]

where $M^{\prime}$ is independent of $Q$.

\end{proof}

If $|| \left\{\frac{1}{|Q|}\int_Q (F^*F) dA(z)\right\}^{\frac{1}{2}}
\left\{\frac{1}{|Q|}\int_Q (F^{-1}F^{*-1})
dA(z)\right\}^{\frac{1}{2}}|| < M$ for all dyadic rectangles $Q$ and
some constant $M$, we will say that $F^*F$ has the matrix $A_2$
condition. See \cite{wavelet} for a similar notion of matrix
weights. We will now find a characterization of such functions $F$
in terms of the boundedness of certain averaging operators on the
function space $L^2(F^*F).$

\begin{thm}
If for $F \in M_{n \times n} (L^2_a)$ the matrix $F^*F$ has the
$A_2,$ condition then the averaging operators, $f \mapsto
\chi_Q\frac{1}{|Q|}\int_Q f(z) dA(z)$, are uniformly bounded on a
dense subset $L^2(\mathbb{C}^n) \cap L^2(F^*F)$, $Q$ varying over
all dyadic rectangles.

\end{thm}
The proof here and of the next theorem follow the reasoning in Lemma
2.1 in \cite{wavelet}.

\begin{proof}
Let $R$ be the subspace $\{\chi_Q\frac{1}{|Q|^{\frac{1}{2}}}
\textbf{e}: \textbf{e} \in \mathbb{C}^n\}$. We can see that the
orthogonal projection from $L^2(\mathbb{D},\mathbb{C}^n)$ onto $R$
is given by $P_Q:f \mapsto \chi_Q\frac{1}{|Q|}\int_Q  f(z) dA(z)$.
So we want to show that these projections are uniformly bounded with
respect to the $L^2(F^*F)$ norm. Clearly,

\[||P_Q||_{L^2(F^*F)} = \sup_{\{f \in L^2 \cap L^2(F^*F): ||f||_{L^2(F^*F)} \neq 0\}}\left\{\frac{||P_Q f
||_{L^2(F^*F)}}{||f||_{L^2(F^*F)}}\right\}
\]

If we let $S$ denote the orthogonal complement of $R$ in $L^2$, then
$f = f_1 + f_2$, where $f_1 \in R$ and $f_2 \in S^\prime = S \cap
L^2(F^*F)$. Thus the expression for the norm of the projection will
become
\begin{multline*}
\sup_{\{f_1+f_2 \in L^2 \cap L^2(F^*F): ||f||_{L^2(F^*F)} \neq
0\}}\left\{\frac{||f_1
||_{L^2(F^*F)}}{||f_1+f_2||_{L^2(F^*F)}}\right\}\\ = \sup_{\{
\textbf{e} \in \mathbb{C}^n: \textbf{e} \neq
0\}}\left\{\frac{||\chi_Q\frac{1}{|Q|^{\frac{1}{2}}}\textbf{e}
||_{L^2(F^*F)}}{\inf_{\{f_2 \in
S\}}||\chi_Q\frac{1}{|Q|^{\frac{1}{2}}}\textbf{e}+f_2||_{L^2(F^*F)}}\right\}
\\= \sup_{\{ \textbf{e} \in \mathbb{C}^n: \textbf{e} \neq
0\}}\left\{\frac{||\chi_Q\frac{1}{|Q|^{\frac{1}{2}}}\textbf{e}
||_{L^2(F^*F)}}{dist_{L^2(F^*F)}\left(\chi_Q\frac{1}{|Q|^{\frac{1}{2}}}\textbf{e},S^\prime\right)}\right\}.
\end{multline*}

So let us take a look at
$dist_{L^2(F^*F)}\left(\chi_Q\frac{1}{|Q|^{\frac{1}{2}}}\textbf{e},S^\prime\right)$.

\begin{multline*}
dist_{L^2(F^*F)}\left(\chi_Q\frac{1}{|Q|^{\frac{1}{2}}}\textbf{e},S^\prime\right)
=
dist_{L^2}\left((F^*F)^{\frac{1}{2}}\chi_Q\frac{1}{|Q|^{\frac{1}{2}}}\textbf{e},(F^*F)^{\frac{1}{2}}S^\prime\right)\\=\sup_{\left\{h
\in
\left((F^*F)^{\frac{1}{2}}S^\prime\right)^\bot:||h||=1\right\}}\left\langle(F^*F)^{\frac{1}{2}}\chi_Q\frac{1}{|Q|^{\frac{1}{2}}}\textbf{e},h\right\rangle,
\end{multline*}
$(F^*F)^{-1}$ exists as we have the $A_2$ condition. Note that
$\left((F^*F)^{\frac{1}{2}}S^\prime\right)^\bot =
\left((F^*F)^{-\frac{1}{2}}R\right).$ Then we can see that:

\[dist_{L^2(F^*F)}\left(\chi_Q\frac{1}{|Q|^{\frac{1}{2}}}\textbf{e},S^\prime\right)=\sup_{\left\{h \in
\left((F^*F)^{-\frac{1}{2}}R\right):||h||=1\right\}}\left\langle(F^*F)^{\frac{1}{2}}\chi_Q\frac{1}{|Q|^{\frac{1}{2}}}\textbf{e},h\right\rangle\]

\[=\sup_{\left\{\textbf{g} \in\mathbb{C}^n:||(F^*F)^{-\frac{1}{2}}\chi_Q\frac{1}{|Q|^{\frac{1}{2}}}\textbf{g}||\leq 1\right\}}\left\langle(F^*F)^{\frac{1}{2}}\chi_Q\frac{1}{|Q|^{\frac{1}{2}}}\textbf{e},(F^*F)^{-\frac{1}{2}}\chi_Q\frac{1}{|Q|^{\frac{1}{2}}}\textbf{g}\right\rangle\]

\[=\sup_{\left\{\textbf{g} \in\mathbb{C}^n:\frac{1}{|Q|}\int_Q\left\langle(F^*F)^{-1}\textbf{g},\textbf{g}\right\rangle\leq 1\right\}}\left\langle(F^*F)^{\frac{1}{2}}\chi_Q\frac{1}{|Q|^{\frac{1}{2}}}\textbf{e},(F^*F)^{-\frac{1}{2}}\chi_Q\frac{1}{|Q|^{\frac{1}{2}}}\textbf{g}\right\rangle\]
\begin{multline*}
=\sup_{\left\{\textbf{g}
\in\mathbb{C}^n:\frac{1}{|Q|}\int_{Q}\left\langle(F^*F)^{-1}\textbf{g},\textbf{g}\right\rangle\leq
1\right\}}\left\langle\chi_Q\frac{1}{|Q|^{\frac{1}{2}}}\textbf{e},\chi_Q\frac{1}{|Q|^{\frac{1}{2}}}\textbf{g}\right\rangle\\
=\sup_{\left\{\textbf{h} \in\mathbb{C}^n:||\textbf{h}||\leq
1\right\}}\left\langle\textbf{e},\left\{\int_{\mathbb{D}}F^{-1}F^{*-1}\chi_Q\frac{1}{|Q|}\right\}^{-\frac{1}{2}}\textbf{h}\right\rangle
\\
=||\left\{\frac{1}{|Q|}\int_{Q}F^{-1}F^{*-1}\right\}^{-\frac{1}{2}}\textbf{e}||
\end{multline*}

Let us now put this equivalent expression for the distance back into
our expression for the norm of the projection in $L^2(F^*F)$;
\begin{multline*}
||P_Q||_{L^2(F^*F) \rightarrow L^2(F^*F)} = \sup_{\{ \textbf{e} \in
\mathbb{C}^n: \textbf{e} \neq
0\}}\left\{\frac{||\chi_Q\frac{1}{|Q|^{\frac{1}{2}}}\textbf{e}
||_{L^2(F^*F)}}{dist_{L^2(F^*F)}\left(\chi_Q\frac{1}{|Q|^{\frac{1}{2}}}\textbf{e},S^\prime\right)}\right\}\\=
\sup_{\{ \textbf{e} \in \mathbb{C}^n: \textbf{e} \neq
0\}}\left\{\frac{||\chi_Q\frac{1}{|Q|^{\frac{1}{2}}}\textbf{e}
||_{L^2(F^*F)}}{||\left\{\frac{1}{|Q|}\int_{Q}F^{-1}F^{*-1}\right\}^{-\frac{1}{2}}\textbf{e}||}\right\}
\end{multline*}

\begin{multline*}= \sup_{\{ \textbf{e} \in \mathbb{C}^n: \textbf{e} \neq
0\}}\left\{\frac{||\left\{\frac{1}{|Q|}\int_Q
F^*F\right\}^{\frac{1}{2}}\textbf{e}
||}{||\left\{\frac{1}{|Q|}\int_{Q}F^{-1}F^{*-1}\right\}^{-\frac{1}{2}}\textbf{e}||}\right\}\\=
\sup_{\{ \textbf{e} \in \mathbb{C}^n: \textbf{e} \neq
0\}}\left\{\frac{||\left\{\frac{1}{|Q|}\int_Q
F^*F\right\}^{\frac{1}{2}}\textbf{e}
||}{||\left\{\frac{1}{|Q|}\int_{Q}F^{-1}F^{*-1}\right\}^{-\frac{1}{2}}\textbf{e}||}\right\}
\end{multline*}

\[=||\left\{\frac{1}{|Q|}\int_Q
F^*F\right\}^{\frac{1}{2}}\left\{\frac{1}{|Q|}\int_{Q}F^{-1}F^{*-1}\right\}^{\frac{1}{2}}
||.
\]
\end{proof}

\begin{lem}

If the averaging operators $g \mapsto \chi_Q\frac{1}{|Q|}\int_Q g(z)
dA(z)$ are uniformly bounded on $L^2(|f|^2)$ over $Q$ dyadic, then
$|f|^2$ has the scalar $A_2$ condition.

\end{lem}

\begin{proof}
Again we can see that the averaging operator $g \mapsto
\chi_Q\frac{1}{|Q|}\int_Q g(z) dA(z)$ is the projection
$P:L^2\rightarrow \chi_Q\frac{1}{|Q|^{\frac{1}{2}}}\mathbb{C}.$ We
are working as before on the dense subset $L^2(\mathbb{C}) \cap
L^2(|f|^2).$ If we assume that $\frac{1}{|f|^2}$ is bounded then we
can as before show that
\begin{multline*}
dist_{L^2(|f|^2)}(\chi_Q\frac{1}{|Q|^{\frac{1}{2}}}z,S^{\prime}) =
\left|\left[\int_{\mathbb{D}}\frac{1}{|f|^2}\chi_Q\frac{1}{|Q|}\right]^{-\frac{1}{2}}z\right|=
\left|\left[\int_{\mathbb{D}}\frac{1}{|f|^2}\chi_Q\frac{1}{|Q|}\right]^{-\frac{1}{2}}\right|.
\end{multline*}
where $|z| = 1$.

So if we drop this assumption on $\frac{1}{|f|^2}$ but instead use
$\frac{1}{|f|^2+\epsilon}$ for $\epsilon > 0$, then we can see that

\[dist_{L^2(|f|^2)}(\chi_Q\frac{1}{|Q|^{\frac{1}{2}}}z,S^{\prime})=
\lim_{\epsilon \rightarrow 0}
dist_{L^2(|f|^2+\epsilon)}(\chi_Q\frac{1}{|Q|^{\frac{1}{2}}}z,S^{\prime})\]

\[=\lim_{\epsilon \rightarrow 0}\left|\left[\frac{1}{|Q|}\int_{Q}\frac{1}{|f|^2+\epsilon}\right]^{-\frac{1}{2}}\right|,\]
where $S^\prime$ is the intersection of the orthogonal complement of
$\chi_Q\frac{1}{|Q|^{\frac{1}{2}}}\mathbb{C}$ with $L^2(|f|^2)$ and
$|z| = 1$.

 As the norm of our bounded projection $P$ is
\[\sup_{z\in\mathbb{C}:z \neq
0, |z| =
1}\left\{\frac{||\chi_Q\frac{1}{|Q|^{\frac{1}{2}}}z||_{L^2(|f|^2)}}{dist_{L^2(|f|^2)}(\chi_Q\frac{1}{|Q|^{\frac{1}{2}}}z,S^{\prime})}\right\}=\sup_{z\in\mathbb{C}:z
\neq 0, |z| =
1}\left\{\frac{||\chi_Q\frac{1}{|Q|^{\frac{1}{2}}}||_{L^2(|f|^2)}}{dist_{L^2(|f|^2)}(\chi_Q\frac{1}{|Q|^{\frac{1}{2}}}z,S^{\prime})}\right\},\]

we know that
$dist_{L^2(|f|^2)}(\chi_Q\frac{1}{|Q|^{\frac{1}{2}}}z,S^{\prime})$
is nonzero for nonzero $z$ and hence \[\lim_{\epsilon \rightarrow
0}\left|\left[\frac{1}{|Q|}\int_{Q}\frac{1}{|f|^2+\epsilon}\right]\right|
< \infty\]

and so by the Monotone Convergence Theorem

\[ \left|\left[\frac{1}{|Q|}\int_{Q}\frac{1}{|f|^2}\right]\right|
< \infty\] and

\[dist_{L^2(|f|^2)}(\chi_Q\frac{1}{|Q|^{\frac{1}{2}}}z,S^{\prime}) =
\left|\left[\frac{1}{|Q|}\int_{Q}\frac{1}{|f|^2}\right]^{-\frac{1}{2}}\right|
\] where $|z| = 1$.

Thus \[||P||_{L^2(|f|^2)} = \sup_{z\in\mathbb{C}:z \neq 0, |z|
=1}\left\{\frac{||\chi_Q\frac{1}{|Q|^{\frac{1}{2}}}z||_{L^2(|f|^2)}}{dist_{L^2(|f|^2)}(\chi_Q\frac{1}{|Q|^{\frac{1}{2}}}z,S^{\prime})}\right\}\]
\[=\sup_{z\in\mathbb{C}:z \neq
0,|z| = 1}\left\{\frac{|\left[\frac{1}{|Q|}\int_Q |f|^2
\right]^{\frac{1}{2}}|}{\left|\left[\frac{1}{|Q|}\int_{Q}\frac{1}{|f|^2}\right]^{-\frac{1}{2}}\right|}\right\}
= \left\{\left|\left[\frac{1}{|Q|}\int_Q |f|^2
\right]^{\frac{1}{2}}\right|\left|\left[\frac{1}{|Q|}\int_{Q}\frac{1}{|f|^2}\right]^{\frac{1}{2}}\right|\right\}
\] which is uniformly bounded as required.

\end{proof}
Compare the next lemma with Lemma 3.6 in \cite{wavelet}.
\begin{lem}
If $F^*F$ has the $A_2$ condition, then $trace(F^*F)$ has the scalar
$A_2$ condition.
\end{lem}
\begin{proof}

We will show that each element on the diagonal of $F^*F$ has the
scalar $A_2$ condition. We can then deduce that the sum of these
will also have the $A_2$ condition. Firstly we know that if $F^*F$
has the $A_2$ condition, then the operators $f \mapsto
\chi_Q\frac{1}{|Q|}\int_Q  f(z) dA(z)$ are uniformly bounded on
$L^2(F^*F),$ $f \in L^2(\mathbb{C}^n)$. So if we take $g \in
L^2(\mathbb{D})\cap L^2(\mathbb{D},\left\langle
F^*F(0,\ldots,1,\ldots,0),(0,\ldots,1,\ldots,0)\right\rangle )$
where $\left\langle
F^*F(0,\ldots,1,\ldots,0),(0,\ldots,1,\ldots,0)\right\rangle$ is the
scalar valued function\\ $z \mapsto \left\langle
F^*(z)F(z)(0,\ldots,1,\ldots,0),(0,\ldots,1,\ldots,0)\right\rangle.$
Then note that $g(0,\ldots,1,\ldots,0) \mapsto
\chi_Q\frac{1}{|Q|}\int_Q g(z)(0,\ldots,1,\ldots,0) dA(z)$ is
uniformly bounded between $L_a^2(\mathbb{D},\mathbb{C}^n)$, this
implies that $g \mapsto \chi_Q\frac{1}{|Q|}\int_Q g(z) dA(z)$ is
uniformly bounded with respect to the scalar measure $\left\langle
F^*F(0,\ldots,1,\ldots,0),(0,\ldots,1,\ldots,0)\right\rangle
_{\mathbb{C}^n}$, which will be whatever diagonal element of $F^*F$
we want. Thus by the previous lemma the trace of $F^*F$ will have
the scalar $A_2$ condition.

\end{proof}

Compare this next lemma with Lemma 4.9 in \cite{products2}, Lemma
2.5 in \cite{invproducts} and also 1.7 on page 196 of \cite{stein}.
\begin{lem}\label{fairshare} If a scalar valued function $|f|^2$ has the $A_2$ condition
and for some $0 < \delta < 1$ then for each dyadic rectangle $Q$ and
$E \subset Q$ such that $|E| \leq \delta |Q|$ we have that $\mu(E)
\leq \lambda\mu(Q)$ for some $0 < \lambda < 1$ where $d\mu = |f|^2
dA$ and $\lambda$ only depends on $\delta$ and the $A_2$ constant of
$|f|^2$.

\end{lem}

\begin{proof}

\[|Q/E|^2 = \left\{\int_{Q/E}|f||f|^{-1}dA\right\}^2 \leq
\left\{\int_{Q/E}|f|^2dA\right\}\left\{\int_{Q/E}|f|^{-2}dA\right\}\]

\[\leq\left\{\int_{Q/E}|f|^2dA\right\}\left\{\int_{Q}|f|^{-2}dA\right\}
\leq
\left\{\int_{Q/E}|f|^2dA\right\}C|Q|^2\left\{\int_{Q}|f|^{2}dA\right\}^{-1}
\]
by our $A_2$ condition on $|f|^2$ this equals

\[(\left\{\int_{Q}|f|^2dA\right\}-\left\{\int_{E}|f|^2dA\right\})C|Q|^2\left\{\int_{Q}|f|^{2}dA\right\}^{-1}
\]\[=
C\left(1-\left\{\int_{E}|f|^2dA\right\}\left\{\int_{Q}|f|^{2}dA\right\}^{-1}\right)|Q|^2,\]

so we know that
\[|Q/E|^2 \leq C|Q|^2\left(1- \frac{\mu(E)}{\mu(Q)}\right)\] and thus

\[ \frac{|Q/E|^2}{|Q|^2} \leq C\frac{\mu(Q/E)}{\mu(Q)}.\]

Now we know that $\frac{|E|}{|Q|} \leq \delta < 1$ from our
hypothesis, this implies that $\frac{|Q/E|}{|Q|} \geq 1 - \delta >
0$. So we can now deduce that

\[ 0 <
\frac{(1 - \delta)^2}{C} \leq \frac{1}{C}\frac{|Q/E|^2}{|Q|^2} \leq
\frac{\mu(Q/E)}{\mu(Q)}.\]

This lets us now see that

\[ 1 = \frac{\mu(Q)}{\mu(Q)} = \frac{\mu(Q/E) + \mu(E)}{\mu(Q)} \geq
\frac{\mu(E)}{\mu(Q)} + \frac{(1-\delta)^2}{C}\] and hence

\[\frac{\mu(E)}{\mu(Q)} \leq 1 - \frac{(1-\delta)^2}{C}.\]

\end{proof}

The following lemma will be crucial to our application of the $A_2$
condition.
\begin{lem}{\label{same}}
If $F^*F$ has the $A_2$ condition and $J$ is a strictly positive
matrix then $JF^*FJ$ will have the $A_2$ condition. The $A_2$
constant of $JF^*FJ$ will depend on the $A_2$ bound of $F^*F$ and
the dimension only.
\end{lem}
\begin{proof}

\[||\left(\frac{1}{|I|}\int_I JF^*FJ\right)^{\frac{1}{2}}\left(\frac{1}{|I|}\int_I
(JF^*FJ)^{-1}\right)^{\frac{1}{2}}||^{2} \]\[=
||\left(\frac{1}{|I|}\int_I
JF^*FJ\right)^{\frac{1}{2}}\left(\frac{1}{|I|}\int_I
(JF^*FJ)^{-1}\right)\left(\frac{1}{|I|}\int_I
JF^*FJ\right)^{\frac{1}{2}}||\]

\[\leq C tr\left(\left(\frac{1}{|I|}\int_I
JF^*FJ\right)^{\frac{1}{2}}\left(\frac{1}{|I|}\int_I
(JF^*FJ)^{-1}\right)\left(\frac{1}{|I|}\int_I
JF^*FJ\right)^{2}\right)\] (where the constant $C$ depends only on
the dimension.)

\begin{multline*}
=C tr\left(\left(\frac{1}{|I|}\int_I
(F^*F)^{-1}\right)\left(\frac{1}{|I|}\int_I F^*F\right)\right)\\\leq
C^\prime ||\left(\frac{1}{|I|}\int_I
F^*F\right)^{\frac{1}{2}}\left(\frac{1}{|I|}\int_I
(F^*F)^{-1}\right)^{\frac{1}{2}}||^{\frac{1}{2}}.
\end{multline*}

$C^\prime$ again depending only on the dimension, thus giving us our
result.

\end{proof}

\begin{defn}
The dyadic maximal operator $M_{\Delta}$ is defined by

\[(M_{\Delta} f)(w) = \sup_{w \in Q}
\frac{1}{|Q|}\int_Q|f(z)|dA(z),\]

where the $Q$ are dyadic rectangles and $f \in L^2$.
\end{defn}

\begin{thm}{(The Calderon-Zygmund Decomposition Theorem.)}\label{calderonzygmund}
Let $f \in L^1(\mathbb{D})$, if we have $t
> 0$ such that the set $\Lambda = \left\{z \in \mathbb{D}: M_{\Delta}f(z)
> t\right\}$ is not the whole of $\mathbb{D},$ then we can decompose
$\Lambda$ into a disjoint union of dyadic intervals $Q_i$ such that
$t < \frac{1}{|Q_i|}\int_{Q_i}|f(z)|dA(z) < 8t.$

\end{thm}

\begin{proof}
The proof of this is exactly as in \cite{invproducts} and
\cite{products2}.
\end{proof}

Compare this next lemma with Proposition 4.14 in \cite{products2}.
\begin{lem}\label{notall}
The trace of $F^*F$ satisies the following;
\begin{enumerate}
    \item \[ tr(F^*F) \leq M_{\Delta} tr(F^*F) \]on $\mathbb{D}$
and
\item \[ \int_{\mathbb{D}}tr(F^*(z)F(z)) dA(z) \leq M_{\Delta}
tr(F^*F)(0) \leq (4/3)^2 \int_{\mathbb{D}}tr(F^*(z)F(z)) dA(z).\]

\end{enumerate}
\end{lem}
\begin{proof}

\begin{enumerate}
\item This follows from Proposition 4.14 in \cite{products2}. We
just need to note that $tr(F^*F)$ is continuous and the proof works
as it is.
\item $\mathbb{D}$ is a dyadic rectangle containing $0$ so
\[M_{\Delta}tr(F^*F)(0) \geq
\frac{1}{|\mathbb{D}|}\int_{\mathbb{D}}tr(F^*(z)F(z)) dA(z) =
\int_{\mathbb{D}}tr(F^*(z)F(z)) dA(z).\]

Let us take a dyadic rectangle $Q$ containing $0$ which is not the
unit disk. We know that $Q$ will be contained in the
pseudohyperbolic disk $D(0,\frac{1}{2}).$

Let $\textbf{e}\in\mathbb{C}^n$, then as $F \in
L^2_a(\mathbb{C}^n),$

\[\left\langle F(u)\textbf{e},F(u)\textbf{e}\right\rangle  =
||F(u)\textbf{e}||^2_{\mathbb{C}^n}  \]

\[=||\int F(z)\overline{K_u(z)}dA(z)\textbf{e}||^2_{\mathbb{C}^n} \leq \left\{\int||F(z)\textbf{e}|||K_u(z)|dA(z)\right\}^2  \]
\[\leq \int \left\langle F\textbf{e},F\textbf{e}\right\rangle  dA(z)||K_u(z)||^2_{L^2} = \left\langle \int F^*(z)F(z)
dA(z) \textbf{e}, \textbf{e}\right\rangle ||K_u(z)||^2_{L^2}.\]

So

\begin{multline*} F^*(u)F(u) \leq \int F^*(z)F(z)dA(z)||K_u||^2_2 \leq \int
F^*(x)F(x)dA(x)\frac{1}{(1-{\frac{1}{2}}^2)^2} \\=
{\left(\frac{4}{3}\right)}^2\int F^*(x)F(x)dA(x)
\end{multline*}

on each $Q$ containing $0$ which is not $\mathbb{D}.$

So \[tr(F^*(u)F(u)) \leq
tr\left({\left(\frac{4}{3}\right)}^2\int_{\mathbb{D}}
F^*(x)F(x)dA(x)\right)\] for $u \in Q.$ Hence

\[  \frac{1}{|Q|} \int_Q
tr(F^*(z)F(z))dA(z) \]\[\leq
{\left(\frac{4}{3}\right)}^2\int_{\mathbb{D}} tr(F^*(x)F(x))dA(x)
\]

and so
\[M_{\Delta}
tr(F^*F)(0) \leq (4/3)^2 \int_{\mathbb{D}}tr(F^*(z)F(z)) dA(z).\]

\end{enumerate}

\end{proof}

The proof of the following theorem follows the lines of of Theorem
2.1 in \cite{invproducts} and Theorem 4.1 in \cite{products2}. It
contains the key to the proof of Theorem \ref{final} i.e. the
reverse H\"{o}lder property.

\begin{thm}\label{revhol}
If $F^*F$ satisfies $A_2$,  then there exists $\epsilon
> 0$ such that $\int (tr(F^*(z)F(z)))^{1+\epsilon}dA(z) \leq C\int
(tr(F^*(z)F(z)))dA(z)^{1+\epsilon}$  with $C$ and $\epsilon$
dependent only on the $A_2$ constant.
\end{thm}
\begin{proof}

For each $k$ define \[E_k = \left\{z \in \mathbb{D}:
M_\Delta(tr(F^*F))(z) >
2^{4k+1}\int_{\mathbb{D}}(tr(F^*(z)F(z)))dA(z)\right\}.\] By Lemma
\ref{notall} we can see that \[M_{\Delta} tr(F^*(0)F(0)) \]\[ \leq
(4/3)^2 \int_{\mathbb{D}}tr(F^*(z)F(z)) dA(z) <
2^{4k+1}\int_{\mathbb{D}}tr(F^*(z)F(z)) dA(z)\] for all $k$. So we
know that each $E_k$ is not the whole disk (as $0$ is not contained
in it) and hence we can do a Calderon-Zygmund decomposition. So for
each $E_k$ we have a disjoint union of dyadic rectangles $Q_i$ whose
union is equal to $E_k$ and
\[2^{4k+1}\int_{\mathbb{D}}(tr(F^*(z)F(z)))dA(z) <
\frac{1}{|Q_i|}\int_{Q_i}tr(F^*(z)F(z)) dA(z) \]\[<
2^{4(k+1)}\int_{\mathbb{D}}(tr(F^*(z)F(z)))dA(z).\]

Two inequalities we will use from this are;

\[|Q_i| < 2^{-4k-1}\left\{\int_{\mathbb{D}}(tr(F^*(z)F(z)))dA(z)\right\}^{-1}\int_{Q_i}tr(F^*(z)F(z))
dA(z)\]

and
\[\int_{Q_i}tr(F^*(z)F(z)) dA(z) <
|Q_i|2^{4(k+1)}\int_{\mathbb{D}}(tr(F^*(z)F(z)))dA(z).\]

We now take a maximal dyadic rectangle $Q$ in $E_{k-1}$ (which is
larger than $E_k$) and note that

\[|E_k \cap Q| = \sum_{Q_i \subset Q}|Q_i|  \](where the $Q_i$ denote the maximal dyadic rectangles in
$E_k$)\[<\sum_{Q_i \subset Q}
2^{-4k-1}\left\{\int_{\mathbb{D}}(tr(F^*(z)F(z)))dA(z)\right\}^{-1}\int_{Q_i}tr(F^*(z)F(z))
dA(z) \]
\[\leq2^{-4k-1}\left\{\int_{\mathbb{D}}(tr(F^*(z)F(z)))dA(z)\right\}^{-1}\int_{Q}tr(F^*(z)F(z))
dA(z)\]

due to the dyadic decomposition of $E_k.$

But as $Q$ is also part of a Calderon-Zygmund decomposition (this
time for $E_{k-1}$) we can also see that

\[\int_{Q}tr(F^*(z)F(z))
dA(z) < |Q|2^{4k}\int_{\mathbb{D}}(tr(F^*(z)F(z)))dA(z).\]

Putting the last two inequalities together we see that

\[|E_k \cap Q|  \]\[<2^{-4k-1}\left\{\int_{\mathbb{D}}(tr(F^*(z)F(z)))dA(z)\right\}^{-1}|Q|2^{4k}\int_{\mathbb{D}}(tr(F^*(z)F(z)))dA(z)
\]\[=
\frac{1}{2}|Q|.\]

We are now in a position to use Lemma \ref{fairshare} as
$tr(F^*(z)F(z)))$ satisfies the scalar $A_2$ condition and $|E_k
\cap Q| \leq \frac{1}{2}|Q|.$ So with $\frac{1}{2}$ being our
$\delta$ in \ref{fairshare}, we can deduce that
\[\mu(E_k\cap Q) < \lambda \mu(Q)\] for some $0 < \lambda < 1$ independent of k, with $d\mu(z)
= tr(F^*(z)F(z)))dA(z).$ We can now sum over all maximal dyadic
rectangles in $E_{k-1}$ and see that

\[\mu(E_k) = \sum_{Q} \mu(E_k \cap Q) < \lambda \sum_{Q} \mu(Q) =
\lambda \mu(E_{k-1}).\]

Let us take a moment here to note that $\lambda$ depends only on our
$A_2$ bound of $tr(F^*(z)F(z))$, (we can see this from Lemma
\ref{fairshare}), and that this $A_2$ bound is controlled by the
matrix $A_2$ bound for $F^*F$ and the dimension.

We have established that for each $k \geq 1$, $\mu(E_k) < \lambda
\mu(E_{k-1})$ and so \[\mu(E_k) < \lambda^k \mu(E_0) = \lambda^k
\int_{E_0} tr(F^*(z)F(z))dA(z)  \leq \lambda^k\int_{\mathbb{D}}
tr(F^*(z)F(z))dA(z).\]

Now let us move on and look at $\int_{\mathbb{D}}
tr(F^*(z)F(z))^{1+\epsilon}dA(z)$ for some $\epsilon > 0.$

From Lemma \ref{notall} we know that $tr(F^*F)(z) \leq M_{\Delta}
tr(F^*F)(z)$ on the disk so

\[\int_{\mathbb{D}}
tr(F^*(z)F(z))^{1+\epsilon}dA(z) \leq \int_{\mathbb{D}}
tr(F^*(z)F(z))\left\{M_{\Delta}
tr(F^*F)(z)\right\}^{\epsilon}dA(z)\]

\[=\int_{x:M_{\Delta} tr(F^*F)(x) \leq \int_{\mathbb{D}}tr(F^*F(z))dA(z)}tr(F^*(z)F(z))\left\{M_{\Delta} tr(F^*F)(z)\right\}^{\epsilon}dA(z)
\]\[+ \sum_k\int_{E_k-E_{k+1}}tr(F^*(z)F(z))\left\{M_{\Delta}
tr(F^*F)(z)\right\}^{\epsilon}dA(z)\]
\[\leq \left\{\int_{\mathbb{D}} tr(F^*F)(z)\right\}^{1+ \epsilon} + \sum_k 2^{(4(k+1) +
1)\epsilon}\left\{\int_{\mathbb{D}}tr(F^*F)(z)dA(z)\right\}^{\epsilon}\mu(E_k)
\]

\[\leq \left\{\int_{\mathbb{D}} tr(F^*F)(z)\right\}^{1+ \epsilon} +\]\[ \sum_k 2^{(4(k+1) +
1)\epsilon}\left\{\int_{\mathbb{D}}tr(F^*F)(z)dA(z)\right\}^{\epsilon}\lambda^k
\int_{\mathbb{D}}tr(F^*F)(z)dA(z) \]

\[=\left\{\int_{\mathbb{D}} tr(F^*F)(z)\right\}^{1+ \epsilon} + \sum_k 2^{(4(k+1) +
1)\epsilon}\left\{\int_{\mathbb{D}}tr(F^*F)(z)dA(z)\right\}^{1+\epsilon}\lambda^k
 \]

\[=\left\{\int_{\mathbb{D}} tr(F^*F)(z)\right\}^{1+ \epsilon} \left(1 +  2^{5\epsilon}\sum_k
(\lambda2^{4\epsilon})^k\right).\]

If we choose $\epsilon$ such that $0 < \lambda 2^{4\epsilon} < 1$
then this will become

\[\left\{\int_{\mathbb{D}} tr(F^*F)(z)\right\}^{1+ \epsilon} \left(1 +  2^{5\epsilon}\frac{1}{1-\lambda2^{4\epsilon}}\right)\]

thus for any $0 < \epsilon^{\prime} \leq \epsilon$ our reverse
H\"{o}lder inequality will hold.

\end{proof}
\begin{cor}{\label{revhol2}}
If $F^*F$ satisfies $A_2$ and $J$ is a positive matrix then there
exists $\epsilon
> 0$ such that $\int (tr(JF^*(z)F(z)J))^{1+\epsilon}dA(z) \leq C\int
(tr(JF^*(z)F(z)J))dA(z)^{1+\epsilon}$. The same $\epsilon$ and
constant $C$ hold for all positive matrices $J$ and $\epsilon$
depends only on the dimension and the $A_2$ constant of $F^*F$.
\end{cor}
\begin{proof}
This follows from \ref{revhol} and \ref{same}.
\end{proof}

\subsection{Proof of Theorem \ref{trace}.}

Two easy lemmas follow before the proof of the Theorem \ref{trace}.

\begin{lem}\label{inverse} Let $F$ and $G$ be matrices
consisting of Bergman space $L_a^2(\mathbb{D})$ functions. If
$FG^*GF^* > \eta I$ and $T_FT_{G^*}$ is bounded, then the Toeplitz
product $T_FT_{G^*}$ is invertible.

\end{lem}

\begin{proof}
$F^*GG^*F > \eta I$ implies that $G^{*-1}F^{-1}F^{*-1}G^{-1}$ is
bounded and so the operator $T_{G^{*-1}F^{-1}} =
T_{G^{*-1}}T_{F^{-1}}$ is bounded. It remains to note that
\[(T_FT_{G^*})T_{G^{*-1}}T_{F^{-1}}F(k_w,0,0,\ldots) =
F(k_w,0,0,\ldots)\] and
\[T_{G^{*-1}}T_{F^{-1}}(T_FT_{G^*})(k_w,0,0,\ldots) =
(k_w,0,0,\ldots),\] and that these also hold for
$(0,\ldots,k_w,\ldots).$ This implication holds because the linear
spans of $\left\{F(0,\ldots,k_w,\ldots)\right\}$ and
$\left\{(0,\ldots,k_w,\ldots)\right\}$ form dense subspaces.
\end{proof}

\begin{lem}
If the trace of a positive matrix $A$ is less than some constant
$\lambda > 0$ then $A < C I$ for some constant $C > 0$ depending
only on $\lambda$ and the dimension, $I$ being the identity matrix.

\end{lem}

\begin{proof}
Trivial.
\end{proof}

\begin{proof}[Proof of Theorem \ref{trace}.]

"$\Leftarrow$"  From Lemma \ref{ap} we know that $F^*F$ satisfies
our $A_2$ condition. Then by Corollary \ref{revhol2},

\begin{equation}\label{oureq}
\int
(tr(((G^*G)(x))^{\frac{1}{2}}(F^*F)(z)((G^*G)(x))^{\frac{1}{2}}))^{1+\epsilon}dA(z)
\leq
\end{equation}
\[C\int
(tr(((G^*G)(x))^{\frac{1}{2}}(F^*F)(z)((G^*G)(x))^{\frac{1}{2}}))dA(z)^{1+\epsilon}
\]
holds for all $x \in \mathbb{D}$ with some $\epsilon > 0$ and a
constant $C$ independent of $x$. Note here that we need to use the
fact that $G^*G$ is strictly positive.

We can also see that $G^*G$ satisfies our $A_2$ condition, so a
similar reverse H\"{o}lder will hold;
\[\int\left(tr\left(\left\{\int(F^*F)(z)dA(z)\right\}^{\frac{1}{2}}(G^*G)(x)\left\{\int(F^*F)(z)dA(z)\right\}^{\frac{1}{2}}\right)\right)^{1+\epsilon'}dA(x)
 \]

\[\leq C\left(\int
\left(tr\left(\left\{\int(F^*F)(z)dA(z)\right\}^{\frac{1}{2}}(G^*G)(x)\left\{\int(F^*F)(z)dA(z)\right\}^{\frac{1}{2}}\right)\right)dA(x)\right)^{(1+\epsilon')}\]

So let us set $\epsilon = \min\left\{\epsilon,\epsilon'\right\}.$

Thus integrating both sides of the reverse H\"{o}lder inequality
(\ref{oureq}) with respect to $x$, we get
\[\int\int
(tr(((G^*G)(x))^{\frac{1}{2}}(F^*F)(z)((G^*G)(x))^{\frac{1}{2}}))^{1+\epsilon}dA(z)dA(x)
 \]

\[\leq C\int\left\{\int
(tr(((G^*G)(x))^{\frac{1}{2}}(F^*F)(z)((G^*G)(x))^{\frac{1}{2}}))dA(z)\right\}^{1+\epsilon}dA(x)\]

\[=C\int
\left(tr\left(((G^*G)(x))^{\frac{1}{2}}\int(F^*F)(z)dA(z)((G^*G)(x))^{\frac{1}{2}}\right)\right)^{1+\epsilon}dA(x)\]

\[=C\int
\left(tr\left(\left\{\int(F^*F)(z)dA(z)\right\}^{\frac{1}{2}}(G^*G)(x)\left\{\int(F^*F)(z)dA(z)\right\}^{\frac{1}{2}}\right)\right)^{1+\epsilon}dA(x)\]

and so as $G^*G$ also has the $A_2$ condition, we can use our
reverse H\"{o}lder again to see that this last expression is less
than or equal to

\[C\left\{\int
\left(tr\left(\left\{\int(F^*F)(z)dA(z)\right\}^{\frac{1}{2}}(G^*G)(x)\left\{\int(F^*F)(z)dA(z)\right\}^{\frac{1}{2}}\right)\right)dA(x)\right\}^{1+\epsilon},\]

where as usual $C$ is a constant that possibly changes from line to
line.

By the M\"{o}bius invariance of the Berezin transform  (\cite{zhu}
page 143) we see that

\[\int\int
(tr(((G^*G)(x))^{\frac{1}{2}}(F^*F)(z)((G^*G)(x))^{\frac{1}{2}}))^{1+\epsilon}|k_w(x)|^2|k_w(z)|^2dA(z)dA(x)
 \]

\[\leq C
(tr((B(G^*G)(w))^{\frac{1}{2}}B(F^*F)(w)(B(G^*G)(w))^{\frac{1}{2}}))^{1+\epsilon}
< CM^{1+\epsilon}.\] Hence by Theorem \ref{eps}, we can see that the
Toeplitz product $T_FT_{G^*}$ is bounded. The invertibility of this
Toeplitz product follows from Lemma \ref{inverse}.
\\
\\
"$\Rightarrow$" If $T_FT_{G^*}$ is bounded and invertible, we know
from Theorem \ref{trace} that \\$tr(B(F^*F)(w)B(G^*G)(w))$ is
uniformly bounded and that $T_FT_{G^*}$ is bounded below. Thus in
particular \[\int\left\langle T_FT_{G^*}k_w
\textbf{e},T_FT_{G^*}k_w\textbf{e}\right\rangle dA(z) > \eta
\int\left\langle k_w \textbf{e}, k_w \textbf{e}\right\rangle dA(z) =
\eta\left\langle \textbf{e},\textbf{e}\right\rangle \] for all
vectors $\textbf{e} \in \mathbb{C}^n$. We know that $T_FT_{G^*}k_w =
F(z)G^*(w)k_w(z)$ and so we deduce that $G(w)B(F^*F)(w)G^*(w) > \eta
I$. From the fact that $||(T_FT_{G^*})^*||$ is also bounded below we
can see that $F(w)B(G^*G)(w)F^*(w) > \eta I$. From these we deduce
the following;

\[B(G^*G)(w) > \eta F^{-1}(w)F^{*-1}(w)\] and

\[B(F^*F)(w) > \eta G^{-1}G^{*-1}(w)\]

which lets us see that
\begin{multline*}
\left\{G^{-1}G^{*-1}(w)\right\}^{\frac{1}{2}}B(G^*G)(w)\left\{G^{-1}G^{*-1}(w)\right\}^{\frac{1}{2}}
\\>
\eta\left\{G^{-1}G^{*-1}(w)\right\}^{\frac{1}{2}}F^{-1}(w)F^{*-1}(w)\left\{G^{-1}G^{*-1}(w)\right\}^{\frac{1}{2}}
\end{multline*}
and also
\begin{multline*}
\left\{B(G^*G)(w)\right\}^{\frac{1}{2}}B(F^*F)(w)\left\{B(G^*G)(w)\right\}^{\frac{1}{2}}
\\>
\eta\left\{B(G^*G)(w)\right\}^{\frac{1}{2}}G^{-1}(w)G^{*-1}(w)\left\{B(G^*G)(w)\right\}^{\frac{1}{2}},
\end{multline*}
 thus

\[tr(B(G^*G)(w)B(F^*F)(w)) =
tr(\left\{B(G^*G)(w)\right\}^{\frac{1}{2}}B(F^*F)(w)\left\{B(G^*G)(w)\right\}^{\frac{1}{2}})
 \]
\begin{multline*}
>\eta
tr(\left\{B(G^*G)(w)\right\}^{\frac{1}{2}}G^{-1}G^{*-1}\left\{B(G^*G)(w)\right\}^{\frac{1}{2}}))\\
=\eta(tr(\left\{G^{-1}G^{*-1}(w)\right\}^{\frac{1}{2}}B(G^*G)(w)\left\{G^{-1}G^{*-1}(w)\right\}^{\frac{1}{2}}))
\\>\eta^2(tr(\left\{G^{-1}G^{*-1}(w)\right\}^{\frac{1}{2}}F^{-1}(w)F^{*-1}(w)\left\{G^{-1}G^{*-1}(w)\right\}^{\frac{1}{2}})).
\end{multline*}

Thus as $tr(B(G^*G)(w)B(F^*F)(w))$ is uniformly bounded,\\
$tr(G^{*-1}(w)F^{-1}(w)F^{*-1}(w)G^{-1}(w))$ is uniformly bounded,
by $\lambda$ ,say, and so\\ $G^{*-1}(w)F^{-1}(w)F^{*-1}(w)G^{-1}(w)
< \lambda^\prime I$, which gives us that $F(w)G^*(w)G(w)F^*(w) >
\frac{1}{\lambda^\prime}I.$

\end{proof}

\subsection*{Acknowledgment}
I wish to thank S. Pott for introducing me to this problem and
discussing various ideas involved.

\end{document}